\documentclass{amsart}

\usepackage[T1]{fontenc}
\usepackage[utf8]{inputenc}
\usepackage{microtype}
\usepackage[colorlinks,citecolor=blue,urlcolor=blue]{hyperref}

\usepackage[backend=bibtex,doi=true,isbn=false,url=false,
            style=alphabetic,maxbibnames=99,maxalphanames=99]{biblatex}
\usepackage{amsthm,amsmath,amsfonts,amssymb}
\usepackage{mathtools}
\usepackage{nicefrac}
\usepackage{graphicx}
\usepackage[config, labelfont={normalsize}]{caption,subfig}
\usepackage[backgroundcolor=yellow]{todonotes}
\usepackage{pgfplots}
\usetikzlibrary{colorbrewer}
\usepgfplotslibrary{colorbrewer}
\usetikzlibrary{decorations.pathreplacing}

\colorlet{basinA}{Paired-B}  %
\colorlet{basinB}{Paired-F}  %
\colorlet{clrA}{Paired-F}    %
\colorlet{clrB}{Paired-B}    %
\colorlet{clrC}{Paired-D}    %

\tikzset{
  edgeA/.style={clrA, line width=1.2pt,
    densely dashed},
  edgeB/.style={clrB, line width=1.8pt, solid},
  edgeC/.style={clrC, line width=1.2pt, dotted},
  topnode/.style={circle, draw, fill=black,
    inner sep=0pt, minimum size=4pt},
  botnode/.style={circle, draw, thick,
    inner sep=0pt, minimum size=8pt},
}

\newcommand{\edgestyle}[1]{%
  \ifcase#1\or edgeA\or edgeB\or edgeC\fi
}
\newcommand{\wallfill}[1]{%
  \ifcase#1\or clrA!20\or clrB!20\or clrC!20\fi
}
\newcommand{\coloringdiagram}[7]{%
  \begin{tikzpicture}[
    baseline=(current bounding box.center),
  ]
    \foreach \i in {1,...,6} {
      \node[topnode] (t\i)
        at ({(\i-1)*0.7}, 2.0) {};
      \node[above=2pt, font=\tiny] at (t\i)
        {$\i$};
    }
    \node[botnode, fill=\wallfill{1}] (b1)
      at (0.7, 0) {};
    \node[below=3pt, font=\footnotesize] at (b1)
      {$1$};
    \node[botnode, fill=\wallfill{2}] (b2)
      at (1.75, 0) {};
    \node[below=3pt, font=\footnotesize] at (b2)
      {$2$};
    \node[botnode, fill=\wallfill{3}] (b3)
      at (2.8, 0) {};
    \node[below=3pt, font=\footnotesize] at (b3)
      {$3$};
    \draw[\edgestyle{#1}] (t1) -- (b#1);
    \draw[\edgestyle{#2}] (t2) -- (b#2);
    \draw[\edgestyle{#3}] (t3) -- (b#3);
    \draw[\edgestyle{#4}] (t4) -- (b#4);
    \draw[\edgestyle{#5}] (t5) -- (b#5);
    \draw[\edgestyle{#6}] (t6) -- (b#6);
    \node[font=\footnotesize] at (1.75, -0.8)
      {$c = #7$};
  \end{tikzpicture}%
}

\newcommand{\coloringdiagramTwo}[5]{%
  \begin{tikzpicture}[
    baseline=(current bounding box.center),
  ]
    \foreach \i in {1,...,4} {
      \node[topnode] (t\i)
        at ({(\i-1)*0.7}, 2.0) {};
      \node[above=2pt, font=\tiny] at (t\i)
        {$\i$};
    }
    \node[botnode, fill=\wallfill{1}] (b1)
      at (0.35, 0) {};
    \node[below=3pt, font=\footnotesize] at (b1)
      {$1$};
    \node[botnode, fill=\wallfill{2}] (b2)
      at (1.75, 0) {};
    \node[below=3pt, font=\footnotesize] at (b2)
      {$2$};
    \draw[\edgestyle{#1}] (t1) -- (b#1);
    \draw[\edgestyle{#2}] (t2) -- (b#2);
    \draw[\edgestyle{#3}] (t3) -- (b#3);
    \draw[\edgestyle{#4}] (t4) -- (b#4);
    \node[font=\footnotesize] at (1.05, -0.8)
      {$c = #5$};
  \end{tikzpicture}%
}

\makeatletter
\@ifclasswith{amsart}{externalize}%
{
  \usetikzlibrary{external}
  \tikzexternalize[
    optimize=false,
    prefix=tikz/,
    mode=list and make]
  \makeatletter
  \renewcommand{\todo}[2][]{\tikzexternaldisable\@todo[#1]{#2}\tikzexternalenable}
  \makeatother
}%
{
  \newcommand{\tikzexternaldisable}{}
  \newcommand{\tikzexternalenable}{}
}
\makeatother

\usepackage{subfiles}
\usepackage{aliascnt}
\usepackage{enumitem}
\usepackage[capitalize,noabbrev]{cleveref}

\numberwithin{equation}{section}

\theoremstyle{definition}
\newtheorem{definition}{Definition}[section]

\theoremstyle{plain}
\newaliascnt{theorem}{definition}
\newtheorem{theorem}[theorem]{Theorem}
\aliascntresetthe{theorem}

\newaliascnt{proposition}{definition}
\newtheorem{proposition}[proposition]{Proposition}
\aliascntresetthe{proposition}

\newaliascnt{lemma}{definition}
\newtheorem{lemma}[lemma]{Lemma}
\aliascntresetthe{lemma}

\newaliascnt{corollary}{definition}
\newtheorem{corollary}[corollary]{Corollary}
\aliascntresetthe{corollary}

\theoremstyle{remark}
\newaliascnt{example}{definition}
\newtheorem{example}[example]{Example}
\aliascntresetthe{example}

\newaliascnt{remark}{definition}
\newtheorem{remark}[remark]{Remark}
\aliascntresetthe{remark}

\newaliascnt{conjecture}{definition}

\aliascntresetthe{conjecture}

\crefname{theorem}{Theorem}{Theorems}
\Crefname{theorem}{Theorem}{Theorems}
\crefname{lemma}{Lemma}{Lemmas}
\Crefname{lemma}{Lemma}{Lemmas}
\crefname{corollary}{Corollary}{Corollaries}
\Crefname{corollary}{Corollary}{Corollaries}
\crefname{proposition}{Proposition}{Propositions}
\Crefname{proposition}{Proposition}{Propositions}
\crefname{definition}{Definition}{Definitions}
\Crefname{definition}{Definition}{Definitions}
\crefname{example}{Example}{Examples}
\Crefname{example}{Example}{Examples}
\crefname{remark}{Remark}{Remarks}
\Crefname{remark}{Remark}{Remarks}
\crefname{conjecture}{Conjecture}{Conjectures}
\Crefname{conjecture}{Conjecture}{Conjectures}

\newcommand{\ZZ}{\mathbb{Z}}
\newcommand{\RR}{\mathbb{R}}
\newcommand{\PP}{\mathbb{P}}
\newcommand{\EE}{\mathbb{E}}
\DeclareMathOperator{\sgn}{sgn}

\DeclareMathOperator{\erfc}{erfc}

\DeclareMathOperator{\Pf}{Pf}
\DeclareMathOperator{\Tr}{Tr}
\DeclareMathOperator{\Var}{Var}
\DeclareMathOperator{\Cov}{Cov}

\newcommand{\ZZhalf}{\ZZ'}

\title%
[Pfaffian structure of basin walls for coalescing particles]%
{Pfaffian structure of basin walls \\
 for coalescing particles}

\author{Piotr \'Sniady}
\address{Institute of Mathematics, Polish Academy of Sciences,
         ul.~\'Sniadeckich 8, \mbox{00-656~Warszawa,} Poland}
\email{psniady@impan.pl}

\begin{document}

\begin{abstract}
Coalescing particles on a line merge when they meet.
As they do, their basins of attraction merge and the
walls between basins disappear. If every site
is initially occupied, these walls at any positive
time form a \emph{Pfaffian point process}: all
correlation functions are determined by pairwise
quantities arranged in antisymmetric matrices.
Tribe, Zaboronski, Garrod, and Poplavskyi established
this structure using analytic methods for
time-homogeneous dynamics; our combinatorial approach
works for any skip-free process (one where particles
cannot change order without first meeting).
We show that the Pfaffian structure lives naturally at
the wall level: we prove an exact
Pfaffian empty-interval formula for the
walls and compute the cumulants of the wall
indicators (higher-order analogs of the variance) as
signed sums of probabilities that independent particles
started at the interval endpoint positions reorder
in specific ways. A structural property of
these sums, indecomposability---every nonzero term
couples all wall positions together---implies a
central limit theorem for the wall count.
A checkerboard duality identifies the walls of one
process with the surviving particles of the dual
process. This covers totally asymmetric dynamics and
position-dependent transition rules, and for Brownian
motion recovers the known Pfaffian point process.
\end{abstract}



\maketitle

\section{Introduction}
\label{sec:introduction}
\subsection{Coalescence and basin walls}
\label{sec:intro-setting}

\subsubsection{Coalescence}

When identical particles on a line collide, they merge
and continue as one
(\Cref{fig:coalescence-intro}). On a lattice,
coalescence describes the voter
model~\cite{HolleyLiggett1975}; in continuous space,
the $A + A \to A$ reaction-diffusion
system~\cite{DoeringBenAvraham1988}. For coalescing
Brownian motions starting from the \emph{maximal
entrance law} (every point initially occupied), the
system comes down from infinity: at any positive time,
only finitely many survivors remain per bounded
interval (Arratia~\cite{Arratia1979}).

\begin{figure}[t]
\centering

\begin{tikzpicture}[scale=0.75]

\def\W{11.5}  %
\def\H{6.0}   %

\draw[gray, thick] (-0.3, 0) -- (\W, 0);
\draw[gray, thick] (-0.3, \H) -- (\W, \H);

\node[gray] at (-0.8, 0) {$\cdots$};
\node[gray] at (\W + 0.5, 0) {$\cdots$};
\node[gray] at (-0.8, \H) {$\cdots$};
\node[gray] at (\W + 0.5, \H) {$\cdots$};

\draw[->, thick, black] (-1.2, 0.4) -- (-1.2, \H - 0.4);
\node[rotate=90, black] at (-1.6, \H/2) {time};

\node[right, font=\small] at (\W + 0.9, 0) {$t = 0$};
\node[right, font=\small] at (\W + 0.9, \H) {$t > 0$};

\draw[black!50, line width=0.4pt]
  (0.0, 0.0) -- (0.38, 0.75) -- (0.0, 1.5) -- (0.38, 2.25);
\draw[black!50, line width=1.0pt]
  (0.38, 2.25) -- (0.75, 3.0) -- (1.12, 3.75) -- (1.5, 4.5);
\draw[black!50, line width=1.4pt]
  (1.5, 4.5) -- (1.12, 5.25) -- (1.5, 6.0);
\draw[black!30, line width=0.4pt]
  (0.75, 0.0) -- (1.12, 0.75) -- (0.75, 1.5) -- (0.38, 2.25);
\draw[black!30, line width=0.4pt]
  (1.5, 0.0) -- (1.12, 0.75);
\draw[black!30, line width=0.4pt]
  (2.25, 0.0) -- (2.62, 0.75) -- (2.25, 1.5) -- (2.62, 2.25)
  -- (2.25, 3.0) -- (1.88, 3.75) -- (1.5, 4.5);

\draw[basinA, line width=0.4pt]
  (3.0, 0.0) -- (3.38, 0.75);
\draw[basinA, line width=0.7pt]
  (3.38, 0.75) -- (3.75, 1.5) -- (4.12, 2.25);
\draw[basinA, line width=1.4pt]
  (4.12, 2.25) -- (4.5, 3.0) -- (4.12, 3.75) -- (3.75, 4.5)
  -- (3.38, 5.25) -- (3.0, 6.0);
\draw[basinA!40, line width=0.4pt]
  (3.75, 0.0) -- (3.38, 0.75);
\draw[basinA!40, line width=0.4pt]
  (4.5, 0.0) -- (4.88, 0.75) -- (4.5, 1.5) -- (4.12, 2.25);
\draw[basinA!40, line width=0.4pt]
  (5.25, 0.0) -- (4.88, 0.75);
\draw[basinA!40, line width=0.4pt]
  (6.0, 0.0) -- (5.62, 0.75) -- (5.25, 1.5) -- (4.88, 2.25)
  -- (4.5, 3.0);

\draw[basinB, line width=0.4pt]
  (6.75, 0.0) -- (7.12, 0.75) -- (7.5, 1.5) -- (7.88, 2.25);
\draw[basinB, line width=1.4pt]
  (7.88, 2.25) -- (7.5, 3.0) -- (7.12, 3.75) -- (6.75, 4.5)
  -- (7.12, 5.25) -- (7.5, 6.0);
\draw[basinB!40, line width=0.4pt]
  (7.5, 0.0) -- (7.88, 0.75) -- (8.25, 1.5) -- (7.88, 2.25);
\draw[basinB!40, line width=0.4pt]
  (8.25, 0.0) -- (8.62, 0.75) -- (8.25, 1.5);
\draw[basinB!40, line width=0.4pt]
  (9.0, 0.0) -- (8.62, 0.75);

\draw[black!50, line width=0.4pt]
  (9.75, 0.0) -- (10.12, 0.75);
\draw[black!50, line width=0.7pt]
  (10.12, 0.75) -- (10.5, 1.5) -- (10.12, 2.25)
  -- (9.75, 3.0) -- (10.12, 3.75) -- (10.5, 4.5)
  -- (10.88, 5.25) -- (11.25, 6.0);
\draw[black!30, line width=0.4pt]
  (10.5, 0.0) -- (10.12, 0.75);

\foreach \px in {0.00, 0.75, 1.50, 2.25, 3.00, 3.75, 4.50,
                 5.25, 6.00, 6.75, 7.50, 8.25, 9.00, 9.75, 10.50} {
  \fill (\px, 0) circle (1.5pt);
}

\foreach \bx in {2.62, 6.38, 9.38} {
  \fill (\bx, 0.30) -- ++(-0.22, -0.42) -- ++(0.44, 0) -- cycle;
}

\foreach \hx in {1.50, 3.00, 7.50, 11.25} {
  \fill (\hx, \H) circle (2.5pt);
}

\end{tikzpicture}

\caption{Coalescing random walks starting from every site.
Paths merge on contact; line weight increases with each merger.
Walls (triangles) mark the boundaries between basins of
attraction at~\mbox{$t = 0$}; survivors (large dots) are the
particles that remain at~$t > 0$, one per basin.}
\label{fig:coalescence-intro}
\end{figure}

Tribe and Zaboronski~\cite{TZ2011} proved that for
coalescing Brownian motions under the maximal
entrance law, the surviving particle positions form a
\emph{Pfaffian point process}; Garrod, Poplavskyi,
Tribe, and Zaboronski~\cite{GarrodPTZ2018} extended
this to continuous-time random walks on~$\ZZ$ with
spatially inhomogeneous rates and all deterministic
initial conditions. In a Pfaffian point process, all
correlation functions are Pfaffians of $2 \times 2$
matrix kernels. A Pfaffian is a signed sum over
perfect matchings of the rows and columns of an
antisymmetric matrix---the analogue of a determinant
when the underlying combinatorics involves pairings
rather than permutations.

\subsubsection{The wall perspective}

This paper takes seriously the ``dual particles''
of Garrod, Poplavskyi, Tribe, and
Zaboronski~\cite[Remark after
Lemma~5]{GarrodPTZ2018}, interprets their kernel as
a crossing-or-meeting probability of independent
particles, and develops the resulting wall
perspective: the Pfaffian structure belongs naturally
not to the surviving particles but to the
\emph{walls}---the boundaries between the
particles' basins of attraction (called partition
points by Arratia~\cite{Arratia1979},
see \Cref{fig:coalescence-intro}).
As particles coalesce,
their basins merge and walls disappear; each
surviving wall contributes a $2 \times 2$ block
to the Pfaffian matrix. The surviving
particles inherit the Pfaffian structure only because
the \emph{checkerboard duality} identifies them with
walls of the dual process
(\Cref{sec:intro-duality}).

\subsubsection{Scope}

The wall perspective yields more than the Pfaffian
point process. From the empty-interval formula
(\Cref{thm:intro-pfaffian} below) we derive an
explicit formula for the cumulants of the wall
count, and a structural property of this formula
(indecomposability) that implies a central limit
theorem. The results also extend to discrete-time systems
with arbitrary inhomogeneous transition probabilities,
including
totally asymmetric dynamics.
They work for any \emph{skip-free} process
(transitions only to neighboring states, so that
particles cannot change order without first
meeting~\cite{KM1959}) and any deterministic initial
condition, but treat only pure coalescence (not the
mixed coalescence-annihilation models
of~\cite{GarrodPTZ2018}).

\subsection{Results for walls}
\label{sec:intro-results}

\subsubsection{The empty-interval formula}

A wall is a boundary between two adjacent basins of
attraction---a half-integer in
$\ZZhalf = \ZZ + \tfrac{1}{2}$ for processes
on~$\ZZ$, or a real number for processes on~$\RR$.

\begin{theorem}[Pfaffian empty-interval formula]
\label{thm:intro-pfaffian}
Consider a coalescing skip-free process with every
point initially occupied. Run the system for
time~$T > 0$, and let
$a_1 < b_1 \leq a_2 < \cdots
\leq a_n < b_n$ be given. Then
\[
\PP\big(\text{no wall in any
  $(a_i, b_i)$}\big) = \Pf(A),
\]
where $A$ is the $2n \times 2n$ antisymmetric matrix
indexed by the $2n$ endpoints
$a_1, b_1, a_2, b_2, \ldots, a_n, b_n$
(numbered $1, 2, \ldots, 2n$ in this order).
For $k < l$, the entry~$A_{kl}$ is the probability
that two independent copies of the underlying
process, started at the $k$-th and $l$-th
endpoints and run for time~$T$, \emph{cross or
meet} (their ordering reverses or they occupy the
same position).
\end{theorem}

Specific instances include:
\begin{itemize}
\item coalescing Brownian motions on~$\RR$
  (\Cref{sec:bm-example});
\item symmetric random walks (Arratia's original
  setting; \Cref{ex:symmetric});
\item totally asymmetric dynamics, where particles
  jump only in one direction
  (\Cref{ex:asymmetric});
\item continuous-time Poisson jumps with space-time
  varying rates (\Cref{ex:continuous-time}).
\end{itemize}

\subsubsection{Pfaffian point process on the lattice}

On the discrete lattice, the empty-interval formula
directly implies that the wall configuration is a
Pfaffian point process: occupancy at
site~$x \in \ZZhalf$ indicates the presence of a
wall in the unit
interval~$(x - \tfrac{1}{2},\, x + \tfrac{1}{2})$,
and inclusion-exclusion converts the
empty-interval Pfaffian into a $2 \times 2$ matrix
kernel whose entries are discrete differences of the
crossing-or-meeting probabilities
(\Cref{prop:discrete-pp}).

\subsubsection{Brownian motion}

For coalescing Brownian motions, the
crossing-or-meeting probabilities~$A_{kl}$ are
complementary error functions of the endpoint
separations (\Cref{sec:bm-example}).
Differentiating the Pfaffian empty-interval formula
and taking the zero-length limit recovers the
Tribe--Zaboronski Pfaffian point process
(\Cref{sec:pfaffian}), the continuous analogue of
the discrete kernel above.

\subsubsection{Cumulants and the central limit theorem}

The Pfaffian formula also yields explicit
combinatorial expressions for the \emph{cumulants} of
the wall count---the connected part of the $k$-point
moment (\Cref{thm:coloring-formula}). These
expressions cleanly separate combinatorics from
probability: the cumulant is a weighted average of
universal integer coefficients~$c(I)$ over
random orderings of independent particles started
at the interval endpoints; the coefficients depend
only on the ordering pattern, not on the process.

A key structural property is
\emph{indecomposability}: every nonzero term in the
$k$-th cumulant couples all $k$ wall positions
together, preventing the sum from splitting into
independent pieces. This explains \emph{why} the wall
process has only local correlations, gives the sign
of the covariance (repulsion) for free, and yields
the cumulant bound $\kappa_k(N_L) = O(L)$ for
$k \geq 2$ (\Cref{thm:clt}).

\subsection{Why Pfaffians?}
\label{sec:intro-why-pfaffians}

\subsubsection{Three links}

Why should the wall process be Pfaffian? The proof
of \Cref{thm:intro-pfaffian} has three links. 

First,
no wall in an interval~$(a,b)$ means that the
particles starting at its endpoints~$a$ and~$b$ have met---a
pairwise coalescence event. 

Second, the \emph{cancellative labeling}
converts pairwise coalescence into annihilation
(\Cref{sec:cancellative-labeling}). 

Third, annihilation
naturally involves matchings: $2k$ particles must
pair up to completely annihilate
($A + A \to \emptyset$), and the different pairings
are exactly the perfect matchings of
$\{1, \ldots, 2k\}$---the combinatorial objects
that define a Pfaffian. The particle trajectories
interact (annihilation!), so the total annihilation
probability does not decompose into pairwise terms
in any obvious way. The companion
paper~\cite{Sniady2026annihilation} proves that the
total annihilation probability nonetheless equals a
Pfaffian---a signed sum over matchings of products
of pairwise crossing probabilities---under the bare
Karlin--McGregor assumptions (order preservation and
the strong Markov property), with no lattice or
generator needed. This single result replaces the
PDE uniqueness argument
of~\cite{TZ2011,GarrodPTZ2018} and applies uniformly
to discrete and continuous processes.

\subsubsection{Connection with Markov duality}

The above three-step reduction---from the interacting
coalescing system to crossing probabilities of
independent particles---has the same structure as a
Markov duality. Garrod, Poplavskyi, Tribe, and
Zaboronski~\cite[Remark after
Lemma~5]{GarrodPTZ2018} observed that their
spin-pair duality fits the standard Markov duality
framework \cite[Chapter~4]{EthierKurtz1986}, with
the dual process being instantly annihilating
particles, and that the same dual process governs
both coalescing and annihilating
systems---suggesting a common origin for their
Pfaffian structure. The cancellative
labeling makes this precise: coalescence reduces to
annihilation, so both are Pfaffian for the same
reason. The annihilation
formula~\cite{Sniady2026annihilation} then decomposes
the duality into two explicit combinatorial steps,
replacing the generator computation by a
sign-reversing involution and extending the result to
inhomogeneous discrete systems where no generator is
available.

\subsection{From walls to particles:
  the checkerboard duality}
\label{sec:intro-duality}

The results above describe walls. To translate
them into statements about surviving
particles---the objects most directly measured in
applications---requires a duality on the planar
lattice.

The duality builds on the graphical construction of
Harris~\cite{Harris1978} and Arratia's ``percolation
substructure''~\cite{Arratia1979}
(\Cref{fig:checkerboard,sec:checkerboard-decomposition}).
The independent random choices that define the
coalescing system simultaneously generate two
complementary non-crossing forests on interleaved
sublattices: a \emph{backward opinion forest}
carrying coalescent lineages, and a \emph{forward
boundary forest} carrying the walls. Different
coordinate choices on the lattice yield different
particle dynamics (\Cref{sec:examples}).

\begin{figure}[t]
\centering
\input{figures/checkerboard-duality.tex}
\caption{Checkerboard duality on the $(u,v)$ lattice.
At each $\ZZhalf^2$-vertex, a random binary choice
(copy from West or South) determines two complementary
forests.
The \textbf{backward opinion forest} (double blue
arrows, on~$\ZZhalf^2$) traces ancestral lineages in
the direction of decreasing~$u+v$. The
\textbf{forward boundary forest} (on~$\ZZ^2$) has
thick red arrows at boundary vertices (where
neighboring opinions differ) and thin red arrows at
non-boundary vertices. The thick diagonal line marks
$u + v = 0$, where all opinions are initially
distinct and every $\ZZ^2$-vertex carries a boundary.}
\label{fig:checkerboard}
\end{figure}

The surviving boundaries of the forward process
are precisely the walls of the backward process
(\Cref{sec:checkerboard-duality}). To study the
surviving particles at time~$T$---which are the
walls of the dual backward process---one applies
the wall-level results of
\Cref{sec:walls-pfaffian,sec:cumulants,sec:clt}
to the dual process. The wall-level results require
the maximal entrance law, but a clamping argument
extends the particle-level results to all
deterministic initial conditions
(\Cref{sec:general-initial}).

\subsection{Prior work}
\label{sec:intro-prior-work}

\subsubsection{Analytic approaches}

The analytic methods
of~\cite{TZ2011,GarrodPTZ2018} and the
checkerboard duality all require skip-free
motion, but their scopes differ.

Tribe and Zaboronski~\cite{TZ2011} established the
Pfaffian point process for coalescing Brownian
motions under the maximal entrance law. Garrod,
Poplavskyi, Tribe, and
Zaboronski~\cite{GarrodPTZ2018} extended this to
continuous-time random walks on~$\ZZ$ with spatially
inhomogeneous rates and all deterministic initial
conditions via the spin-pair duality, yielding a
$2 \times 2$ matrix kernel giving all correlation
functions. Tribe and
Zaboronski~\cite{TribeZaboronski2026} further
extended it to every entrance law, again via the
spin-pair duality. FitzGerald, Tribe, and
Zaboronski~\cite{FitzGeraldTZ2022} proved sharp
asymptotics for the Fredholm Pfaffians arising from
these kernels, rigorously confirming the gap
exponents predicted by Derrida and
Zeitak~\cite{DerridaZeitak1996}.

The spin-pair duality and the checkerboard approach
cover complementary territory. The spin-pair duality
treats mixed coalescence-annihilation (with
annihilation probability
$\theta \in [0,1]$)~\cite{GarrodPTZ2018}
and---as shown
in~\cite{TribeZaboronski2026}---all entrance laws,
but requires a time-homogeneous Markov generator
with the right algebraic structure. The present
paper treats pure coalescence with all deterministic
initial conditions, and works on discrete lattices
with arbitrary inhomogeneous transition probabilities,
where no generator or PDE is available. Neither approach
subsumes the other; they overlap for
time-homogeneous random walks and Brownian motion.
In this overlap, the two dualities produce the same
mathematical object: the ``dual particles'' of the
spin-pair duality are the backward particles of the
checkerboard construction
(\Cref{sec:gptz-connection}).

For time-homogeneous coalescing systems, the
Pfaffian point process structure itself is therefore
not new---it was established
in~\cite{TZ2011,GarrodPTZ2018}---and a CLT was
proved by Glinyanaya and
Fomichov~\cite{GlinyanayaFomichov2018} via spatial
mixing. The
contributions of the present paper are: the
extension to time-inhomogeneous and totally
asymmetric dynamics (\Cref{sec:examples}), where
no generator is available; the cumulant coloring
formula (\Cref{thm:coloring-formula}) and its
structural properties, which have no analogue in
the analytic approach; and the structural
explanation for the CLT via indecomposability
(\Cref{thm:clt})---every nonzero cumulant term
couples all wall positions, which forces local
correlations and implies $\kappa_k(N_L) = O(L)$
for all~$k \geq 2$.

\subsubsection{CLT for point processes}

Cumulants of determinantal point processes are
well understood.
Soshnikov~\cite{Soshnikov2002} proved
$\kappa_k = (-1)^{k-1} \Tr(K^k)$ for Hermitian
kernels~$K$, from which the CLT follows whenever the
variance diverges. For \emph{Pfaffian} point
processes, the corresponding theory is more recent.
Wang and Xu~\cite{WangXu2025} proved the trace
formula and a CLT under a finite-rank commutator
property for the operators derived from the
$2 \times 2$ matrix kernel, verified for the
Sine$_1$ and Sine$_4$ processes.
Lin, Qiu, and Wang~\cite{LinQiuWang2021} proved a
CLT for Pfaffian point processes whose $2 \times 2$
matrix kernel satisfies a symmetry condition
($J$-Hermiticity), without explicit cumulant
formulas.

\subsubsection{CLT for coalescing Brownian motions}

Glinyanaya and
Fomichov~\cite{GlinyanayaFomichov2018} proved a CLT
(with Berry--Esseen bound) for the number of
surviving coalescing Brownian motions in a bounded
interval, using spatial mixing (the
Ibragimov--Linnik CLT for weakly dependent
sequences~\cite{IbragimovLinnik1971}). The
coloring formula (\Cref{thm:coloring-formula})
provides a different route to the CLT, exploiting
the probabilistic interpretation of the kernel
entries as crossing-or-meeting probabilities of
independent particles. Indecomposability of the
cumulant expansion ensures convergence of the
integrated cumulant integrals, bypassing the
kernel-specific conditions required by the
approaches above.

\subsubsection{Companion papers}

This paper is part of a series. The coalescence
paper~\cite{Sniady2026coalescence} proves exact
determinantal formulas for coalescence probabilities
on arbitrary planar weighted directed acyclic
graphs. The annihilation
paper~\cite{Sniady2026annihilation} extends this to
annihilation and proves the Pfaffian pairwise
coalescence formula. A third
companion~\cite{Sniady2026coalescenceApplications}
studies the \emph{wall-particle system}---the joint
system of survivors and their
basin boundaries under the maximal entrance law.
The present paper connects the combinatorial
framework to Pfaffian point processes.

\subsection{Organization}
\label{sec:intro-organization}

The paper has two parts. The first
(\Cref{sec:walls-pfaffian,sec:cumulants,sec:clt,sec:empty-to-pp,sec:direct-examples})
develops the wall-level theory without any
lattice structure: the Pfaffian empty-interval
formula (\Cref{sec:walls-pfaffian}), the cumulant
coloring formula and its three structural properties
(\Cref{sec:cumulants}), the central limit theorem
for the wall count (\Cref{sec:clt}), correlation
functions from empty-interval probabilities
(\Cref{sec:empty-to-pp}), and the
Brownian motion example recovering the
Tribe--Zaboronski
kernel~\cite{TZ2011,GarrodPTZ2018}
(\Cref{sec:direct-examples}).
The second part
(\Cref{sec:checkerboard,sec:examples})
translates these results to surviving particles
via the checkerboard duality
(\Cref{sec:checkerboard}) and works out explicit
formulas for several discrete particle systems
(\Cref{sec:examples}).
\Cref{sec:appendix-clt} proves the underlying CLT
for sums of dependent indicators via idempotence,
and \Cref{sec:appendix-colorings} tabulates the
nonzero colorings for $k = 3$.

\section{Pfaffian formula for walls}
\label{sec:walls-pfaffian}

This section proves the Pfaffian empty-interval formula
(\Cref{thm:intro-pfaffian}) for the wall process of any
skip-free coalescing system. The proof uses only the
definition of walls and the cancellative labeling; no
duality or lattice structure is needed.

\subsection{Walls and coalescence}
\label{sec:walls-coalescence}

Consider a coalescing \emph{skip-free} process:
particles move on a linearly ordered space, cannot
change order without first
meeting~\cite{KM1959}, and coalesce when they meet.
On~$\ZZ$ this means nearest-neighbor transitions;
on~$\RR$ it means continuous paths (such as Brownian
motion). The transition probabilities may be
inhomogeneous in both space and time. Assume every
point is initially occupied (the maximal entrance
law). Fix a time~$T > 0$.

The surviving particles partition the line into
\emph{basins of attraction}: the basin of a
survivor~$y$ is the contiguous set of initial
positions whose particle has coalesced with~$y$ by
time~$T$. A \emph{wall} is a boundary between two
adjacent basins (\Cref{sec:intro-setting}).

\begin{proposition}[No wall equals coalescence]
\label{prop:no-wall-coalescence}
For initial positions $a < b$, there is no wall in
the interval~$(a, b)$ if and only if the particles
starting at~$a$ and~$b$ have coalesced by
time~$T$.
\end{proposition}

\begin{proof}
If the particles from~$a$ and~$b$ have coalesced,
they belong to the same survivor. Since the process
is skip-free, every particle starting between~$a$
and~$b$ must also have met them (it cannot escape
without first meeting one of its neighbors). All
points in~$[a, b]$ share the same basin, so no wall
exists in~$(a, b)$.

Conversely, if no wall exists in~$(a, b)$, then
all points in~$[a, b]$ belong to the same basin,
so the particles from~$a$ and~$b$ have coalesced.
\end{proof}

For $n$~pairs $a_1, b_1, \ldots, a_n, b_n$
with $a_1 < b_1 \leq a_2 < \cdots \leq a_n < b_n$,
the event $\{\text{no wall in any $(a_i, b_i)$}\}$
equals \emph{pairwise coalescence}: for each~$i$,
the particles from~$a_i$ and~$b_i$ have coalesced.

\subsection{Cancellative labeling (coalescence to
annihilation)}
\label{sec:cancellative-labeling}

Label each of the $2n$ particles (from positions
$a_1, b_1, \ldots, a_n, b_n$) with~$1$,
and all other particles with~$0$.
When two particles coalesce, let the merged
particle carry the sum of their labels $\bmod\; 2$.
A cluster of~$s$ coalesced particles therefore
carries label~$s \bmod 2$. Tracking only the clusters
with label~$1$ gives an \emph{annihilating} particle
system: when two label-$1$ clusters meet, both labels
become~$0$ ($1 \oplus 1 = 0$). \emph{Total
annihilation} means that no label-$1$ cluster
survives; equivalently, every cluster contains an
even number of labeled particles.
This mod-$2$ rule is the \emph{cancellative labeling}
of Griffeath~\cite{Griffeath1979} (see
also~\cite{AthreyaSwart2012,benAvrahamBrunet2005}).

\begin{proposition}[Pairwise coalescence equals
total annihilation]
\label{prop:parity-trick}
Pairwise coalescence of the $2n$ particles equals
total annihilation under the cancellative labeling.
\end{proposition}

\begin{proof}
Since the process is skip-free, particles that have
met follow the same path thereafter, forming clusters
of colocated particles. Clusters consist of
consecutive original particles (a particle cannot
join a non-adjacent cluster without first meeting its
neighbor).

Under pairwise coalescence, each pair $(a_i, b_i)$
is in the same cluster. Since pair members are
adjacent in the ordering, each cluster is a union
of consecutive pairs and therefore contains an even
number of labeled particles. All labels
are therefore~$0$: total annihilation.

Conversely, if some pair $(a_i, b_i)$ is in
different clusters, the clusters containing
$a_1, b_1, \ldots, a_{i-1}, b_{i-1}, a_i$ hold
$2i - 1$ labeled particles in total. Since labels
add modulo~$2$, the total label of these clusters
is~$1$, so at least one cluster carries
label~$1$: not total annihilation.
\end{proof}

\subsection{Pfaffian empty-interval formula}
\label{sec:pfaffian-empty-interval}

\begin{theorem}[Pfaffian empty-interval formula]
\label{thm:pfaffian-empty-interval}
Consider a coalescing skip-free process with every
point initially occupied. Run the system for
time~$T > 0$, and let
$a_1 < b_1 \leq a_2 < \cdots
\leq a_n < b_n$ be given.
Then:
\[
\PP(\text{no wall in any
  $(a_i, b_i)$}) = \Pf(A),
\]
where $A$ is the $2n \times 2n$ antisymmetric matrix
whose entry $A_{kl}$ ($k < l$) is the probability
that independent particles from
the $k$-th and $l$-th endpoints cross or
meet.
\end{theorem}

\begin{proof}
By \Cref{prop:no-wall-coalescence}, the event
equals pairwise coalescence of the $2n$
particles. By \Cref{prop:parity-trick}, pairwise
coalescence equals total annihilation. The
Pfaffian annihilation
formula~\cite{Sniady2026annihilation} gives the
probability as~$\Pf(A)$.
\end{proof}

\begin{remark}[Touching intervals]
\label{rem:touching-intervals}
When $b_i = a_{i+1}$, two particles start at the
same position. The Pfaffian annihilation
formula~\cite{Sniady2026annihilation} allows
coincident starting positions, so the proof applies
directly: the entry $A_{kl} = 1$ for each
coincident pair (paths from the same position
trivially cross or meet).
\end{remark}

\section{The cumulant coloring formula}
\label{sec:cumulants}

The Pfaffian empty-interval formula
(\Cref{thm:intro-pfaffian}) computes joint
void probabilities for the wall process. This section
extracts the \emph{cumulants}---the connected part of
the $k$-point moment---via a concrete recipe.

To compute the $k$-th mixed cumulant of the wall
indicators, run $2k$ independent particles, two per
interval: one from each endpoint. Call these the
\emph{flanking particles} of the wall. No
coalescence, no interaction between different walls.
Assign \emph{color}~$w$ to both flanking particles
of wall~$w$. At time~$T$, read their colors from
left to right: this random sequence is the
\emph{coloring}~$I$.
The cumulant is a weighted average of a universal
integer coefficient~$c(I)$ over all coloring
outcomes (\Cref{thm:coloring-formula}), but with
one constraint: within each wall's pair, the two
flanking particles must not have crossed
during~$[0, T]$. If they did, the integrand
vanishes. Absorbing this non-crossing
condition into a conditional measure gives a clean
factorization: the cumulant equals a product of
non-crossing probabilities (one per wall) times a
conditional expectation of~$c(I)$
(\Cref{sec:non-colliding-conditioning}).

Three selection rules further reduce the
contributing colorings: the coefficients vanish
unless the coloring is interleaving
(\Cref{thm:proper-coloring}) and indecomposable
(\Cref{thm:indecomposability}), and boundary
colorings also vanish
(\Cref{thm:boundary-vanishing}).
Indecomposability controls the growth of cumulants
with the interval length, leading to the central
limit theorem (\Cref{thm:clt}). The other two
properties further reduce the nonzero terms but
are not needed for the CLT.

\subsection{The coloring formula}
\label{sec:coloring-formula}

\subsubsection{Setup and statement}

Fix $k$ disjoint intervals
$(a_1, b_1), \ldots, (a_k, b_k)$ with
$a_1 < b_1 \leq a_2 < \cdots \leq a_k < b_k$.
On~$\ZZ$, the endpoints $a_w, b_w$ are integers
(so that flanking particles exist at both
endpoints).
Let $\eta_w$ denote the indicator that at least one
wall lies in~$(a_w, b_w)$.

Start $2k$ independent copies of the underlying
process, two per interval: one from~$a_w$ and one
from~$b_w$, each carrying color~$w$. Run them to
time~$T$. Reading the colors of the $2k$ final
particles from left to right (in the discrete
case, breaking ties uniformly at random) gives the
\emph{coloring}~$I$---a sequence of length~$2k$
over the alphabet $\{1, \ldots, k\}$, each letter
appearing exactly twice.

\begin{theorem}[Coloring formula]
\label{thm:coloring-formula}
For $k \geq 2$,
\begin{equation}\label{eq:coloring-formula}
\kappa(\eta_1, \ldots, \eta_k)
= \EE\Bigl[c(I) \cdot \prod_{w=1}^{k}
  \mathbf{1}_{\textup{wall } w
  \textup{ non-crossing}}\Bigr],
\end{equation}
where $c(I)$ is a universal integer depending only on
the coloring (defined below), ``non-crossing'' means
that wall~$w$'s two flanking particles did not cross
during~$[0, T]$, and the expectation is over
the random trajectories and, in the discrete case,
the random tie-breaking.
Here $I$ is the random coloring realized by the
$2k$ particles; $c(I)$ is a deterministic function
of the coloring pattern.
\end{theorem}

\subsubsection{The coloring coefficient}

Define the \emph{wall edges}
$(1, 2), (3, 4), \ldots, (2k{-}1, 2k)$, one per wall.
A perfect matching~$M$ of $\{1, \ldots, 2k\}$ is
\emph{connected} if the union of wall edges and
the edges of $M$ forms a single Hamiltonian cycle.
For $k \geq 2$, no connected matching contains a
wall edge
$(2w{-}1, 2w)$: it would form a $2$-cycle with
the wall edge, breaking the Hamiltonian cycle.
In particular, every matching edge connects
vertices from different walls.

The sequence $(1, 1, 2, 2, \ldots, k, k)$ records
the colors of the $2k$ interval endpoints in their
initial order (wall~$1$'s two endpoints first, then
wall~$2$'s, etc.). The symmetric group~$S_{2k}$
acts on sequences by permuting coordinates. A
\emph{lift of a coloring~$I$} is a
permutation $\sigma \in S_{2k}$ satisfying
$\sigma \cdot (1, 1, 2, 2, \ldots, k, k) = I$.
Each coloring has exactly~$2^k$
lifts (one swap per wall).
Write $\alpha_w < \beta_w$ for the two positions
of color~$w$ in~$I$; then
$\{\sigma(2w{-}1), \sigma(2w)\}
= \{\alpha_w, \beta_w\}$.

A lift~$\sigma$ is \emph{compatible} with a
connected matching~$M$ if $\sigma(i) > \sigma(j)$ for
every edge $\{i,j\}\in M$ with $i < j$.

Two signs enter the formula. The \emph{matching
sign} $\sgn(M)$ is the sign of the Pfaffian
expansion: writing the edges of~$M$ as
$(i_1, j_1), \ldots, (i_k, j_k)$ with $i_r < j_r$
and $i_1 < \cdots < i_k$,
\[
\sgn(M) = \sgn(i_1, j_1, \ldots, i_k, j_k)
\]
as a permutation of $\{1, \ldots, 2k\}$. The
\emph{lift sign}
\begin{equation}\label{eq:epsilon}
\epsilon(\sigma)
= \prod_{w=1}^{k}
  (-1)^{[\sigma(2w-1) > \sigma(2w)]}
\end{equation}
records whether each wall's left flanking particle
lands at the lower or upper rank.

\begin{definition}[Coloring coefficient]
For a coloring~$I$ of length~$2k$, the
\emph{coloring coefficient} is
\begin{equation}\label{eq:coloring-coefficient}
c(I) = (-1)^k \sum_{M \textup{ connected}}
  \sgn(M)
  \sum_{\substack{\sigma \textup{ lift of } I \\
                  \textup{compat.\ with } M}}
  \epsilon(\sigma).
\end{equation}
\end{definition}

\subsubsection{The assumption $k \geq 2$}
The derivation converts wall indicators
$\eta_w = \mathbf{1}_{\text{wall}}$ to
no-wall indicators
$\xi_w = 1 - \eta_w = \mathbf{1}_{\text{no wall}}$
via
$\kappa_k(\eta_1, \ldots, \eta_k)
= (-1)^k \kappa_k(\xi_1, \ldots, \xi_k)$.
This uses shift-invariance of mixed cumulants,
which holds for $k \geq 2$ but fails for $k = 1$:
$\kappa_1(\eta) = 1 - \kappa_1(\xi)
\neq -\kappa_1(\xi)$.
The first cumulant $\kappa_1(\eta) = P(\text{wall
in } (a, b))$ is computed directly from the
Pfaffian for a single interval.

\subsection{Small examples}
\label{sec:small-examples}

\subsubsection{Covariance}
\label{sec:covariance-coloring}
For $k = 2$, the only nonzero coloring is
$I = (2, 1, 2, 1)$, with $c(I) = -4$.
This coloring means the four particles interleave:
a wall-$2$ particle finishes leftmost, then a
wall-$1$ particle, then wall-$2$, then wall-$1$.
Combined with the non-crossing condition for each
wall, the formula gives
\[
\Cov(\eta_1, \eta_2)
= -4 \cdot P\!\left(\!
  \begin{array}{c}
  Y_{a_2} < Y_{a_1} < Y_{b_2} < Y_{b_1} \\
  \text{and both walls non-crossing}
  \end{array}\!\right)\!,
\]
where $Y_s$ denotes the time-$T$ position of an
independent particle started at~$s$, and
``non-crossing'' means the two flanking particles of
each wall never swapped order during~$[0, T]$.
The negative sign reflects repulsion between walls.

\begin{figure}[t]
\centering
\coloringdiagramTwo{2}{1}{2}{1}{-4}
\caption{The unique nonzero coloring for $k = 2$:
  $I = (2, 1, 2, 1)$, $c(I) = -4$.
  Positions $1$ and~$3$ are assigned to wall~$2$
  (dashed), positions $2$ and~$4$ to wall~$1$
  (solid), giving the interleaving pattern.}
\label{fig:k2-coloring}
\end{figure}

\subsubsection{Other small examples}
For $k = 3$, there are $90$ colorings (sequences of
length~$6$ over $\{1, 2, 3\}$ with each letter
twice), of which $14$ have $c(I) \neq 0$.
All nonzero values satisfy $|c(I)| = 8$; the
complete list appears in
\Cref{sec:appendix-colorings}.
Only one coloring has $c < 0$; the rest are positive.
For $k = 4$, there are $386$
nonzero colorings (out of $2520$ total),
with $|c(I)| \in \{16, 32, 48\}$
(all multiples of~$2^k$).

\begin{remark}[Sign of cumulants]
For $k = 2$, the unique nonzero coloring has
$c(I) = -4 < 0$, so the covariance is manifestly
negative (repulsion) for every skip-free process.
For $k \geq 3$, both positive and negative
coefficients appear
(\Cref{sec:appendix-colorings}).
Whether the total cumulant $\kappa_k$ has a
definite sign for all skip-free processes remains
an open question.
\end{remark}

\subsection{Derivation}
\label{sec:cumulants-derivation}

\subsubsection{From Pfaffian to connected matchings}

Assume first that the transition law has a
continuous density~$p(s, y)$ (no atoms), so that
final positions are almost surely distinct; the
discrete case is treated at the end of the
derivation.

Since
$\eta_w = 1 - \mathbf{1}_{\text{no wall in
$(a_w, b_w)$}}$ and cumulants of order $k \geq 2$
are invariant under additive constants, the mixed
cumulants of the wall indicators equal $(-1)^k$
times those of the void indicators.
The void probabilities are Pfaffians
(\Cref{thm:intro-pfaffian}), and
the cumulant is the connected part of the joint
moment~\cite{Speed1983}. For Pfaffians, M\"obius
inversion on the partition lattice selects
precisely the \emph{connected} matchings from the
Pfaffian expansion~\cite{WangXu2025}.

By the Karlin--McGregor segment
swap~\cite{KM1959}, each crossing-or-meeting
probability satisfies
\[
A_{ij} = 2 \int_{z_i > z_j}
  p(s_i, z_i)\, p(s_j, z_j)\, dz_i\, dz_j,
\]
where $s_i, s_j$ are the starting positions:
paths that crossed and ended unswapped biject
with paths that crossed and ended swapped.
Expanding all $k$ matching edges simultaneously,
each connected matching~$M$ produces a
$2k$-dimensional integral over
$z_1, \ldots, z_{2k}$ with the constraint
$z_i > z_j$ for every matching edge
$\{i, j\} \in M$ with $i < j$.

Rewriting on the ordered simplex
$y_1 < y_2 < \cdots < y_{2k}$ via a permutation
$\sigma \in S_{2k}$ (particle~$i$ lands at
position~$y_{\sigma(i)}$), the constraint becomes
$\sigma(i) > \sigma(j)$.
The contribution of each compatible~$\sigma$ is
the simplex integral
\begin{equation}\label{eq:simplex-integral}
J(\sigma)
= \int_{y_1 < \cdots < y_{2k}}
  \prod_{i=1}^{2k} p(s_i, y_{\sigma(i)})\, dy,
\end{equation}
where $s_i$ is the starting position of particle~$i$
($s_{2w-1} = a_w$, $s_{2w} = b_w$) and
$p(s, y)$ the transition density.

\subsubsection{From simplex to coloring integral}

Wall~$w$'s two particles ($2w{-}1$ and $2w$,
starting at~$a_w$ and~$b_w$) contribute the factor
$p(a_w, y_{\sigma(2w-1)}) \cdot
p(b_w, y_{\sigma(2w)})$
to~\eqref{eq:simplex-integral}.
Since $\{\sigma(2w{-}1), \sigma(2w)\}
= \{\alpha_w, \beta_w\}$, each lift~$\sigma$
assigns one of the two flanking particles to
each position. Summing over the two lifts that
differ by swapping within wall~$w$ produces the
\emph{Karlin--McGregor determinant}
\[
\det\nolimits_w(u, v)
\coloneqq
\begin{vmatrix}
p(a_w, u) & p(a_w, v) \\
p(b_w, u) & p(b_w, v)
\end{vmatrix}.
\]
Doing this for all $k$ walls simultaneously, the
product over walls absorbs all lift
signs~$\epsilon(\sigma)$
from~\eqref{eq:epsilon}.
By the Karlin--McGregor theorem~\cite{KM1959},
$\det_w(u,v)\, du\, dv$ is the probability that
two independent particles from~$a_w$ and~$b_w$
end at~$(u,v)$ without crossing; in particular,
$\det_w(u,v) \geq 0$ for $u < v$.
Summing over all connected matchings and compatible
lifts gives the integral form of the
coloring formula:
\begin{equation}\label{eq:coloring-integral}
\kappa(\eta_1, \ldots, \eta_k)
= \sum_I c(I) \int_{y_1 < \cdots < y_{2k}}
  \prod_{w=1}^{k}
  \det\nolimits_w(y_{\alpha_w}, y_{\beta_w})\, dy.
\end{equation}
This is equivalent to
\Cref{thm:coloring-formula}: each
$\det_w(y_{\alpha_w}, y_{\beta_w})\, dy$ is the
probability that wall~$w$'s flanking particles end
at $(y_{\alpha_w}, y_{\beta_w})$ without crossing,
and integrating over the simplex sums exactly
those outcomes where the coloring equals~$I$.

\subsubsection{Discrete tie-breaking}

In the discrete setting, ties are broken by
attaching a private copy of~$\RR$ to each
final vertex~$n$: a particle at discrete position~$n$
receives an independent uniform continuous
coordinate on a small interval confined
to~$n$'s copy (from which it cannot reach
another vertex's interval). This does not alter
the wall system: particles sharing a copy of~$\RR$
have already coalesced in the discrete dynamics,
so they belong to the same basin regardless of
the continuous perturbation.
The lexicographic
order breaks all ties almost surely, so the
strict simplex formulation applies unchanged.
Since the transition probabilities depend only
on the discrete coordinate, integrating out
the continuous tie-breakers recovers
the $1/\prod_j m_j!$ multiplicity weights on
the weak simplex
$y_1 \leq y_2 \leq \cdots \leq y_{2k}$.
When both flanking particles of a wall land at
the same vertex, the $\det_w$ factor vanishes
(the determinant has identical columns), so only
configurations with distinct within-wall
positions contribute.

\subsection{Non-crossing conditioning}
\label{sec:non-colliding-conditioning}

In the coloring formula~\eqref{eq:coloring-formula},
the $k$ pairs of particles evolve independently,
and the non-crossing indicators can be absorbed
into a conditional measure: condition each pair
on its two particles not crossing during~$[0,T]$.
This gives
\[
\kappa(\eta_1, \ldots, \eta_k)
= \prod_{w=1}^k P_w(\text{no crossing})
  \cdot \EE_{\mathrm{cond}}\bigl[c(I)\bigr],
\]
where $P_w(\text{no crossing})$ is the probability
that wall~$w$'s flanking particles maintain their
order, and the expectation is over $k$
independently conditioned non-crossing pairs.

As the flanking distance
$\varepsilon = b_w - a_w \to 0$, the prefactor
$P_w(\text{no crossing}) = O(\varepsilon)$
captures all the vanishing, while the conditional
density~$p_{\mathrm{cond},w}$ of wall~$w$'s pair
given non-crossing converges to the density
of a non-colliding Brownian pair---the Doob
$h$-transform with $h(a, b) = b - a$, which
conditions the pair to maintain strict order
throughout~$[0,T]$. The bridge density of this
conditioned process involves
$\varphi(y-x)$ and
$(y - x)\varphi(y - x)$, which are the row
functions of the $2 \times 2$
kernel~(\Cref{sec:pfaffian}).
The factorization also explains the Fubini step
in the CLT proof (\Cref{sec:clt-proof-discrete}):
integrating $\det_w$ over the wall position
computes $P_w(\text{no crossing})$ marginalized
over the wall, which has Gaussian-type decay in
the gap $y_{\beta_w} - y_{\alpha_w}$.

\subsection{Interleaving colorings}
\label{sec:proper-coloring}

A coloring~$I$ is \emph{interleaving} if no two
adjacent entries are equal.

\begin{theorem}[Only interleaving colorings may contribute]
\label{thm:proper-coloring}
For $k \geq 2$, if~$I$ is not interleaving, then
$c(I) = 0$.
\end{theorem}

\begin{proof}
We show that for each connected matching~$M$, the inner sum
\[
\sum_{\substack{\sigma \textup{ lift of } I \\
                \textup{compat.\ with } M}}
\epsilon(\sigma)
= 0.
\]
Define an involution~$\tau_w$ on lifts
by swapping within wall~$w$: set
$\tau_w(\sigma)(2w{-}1) = \sigma(2w)$ and
$\tau_w(\sigma)(2w) = \sigma(2w{-}1)$, leaving all other
values unchanged.

Since~$I$ is not interleaving, some adjacent positions
carry the same color: $I_r = I_{r+1} = w$ for some~$r$.
Wall~$w$'s two simplex positions are therefore
$\{\alpha_w, \beta_w\} = \{r, r+1\}$, and $\tau_w$
exchanges positions~$r$ and~$r+1$.

\medskip

\emph{Sign flip.}
$\epsilon(\tau_w(\sigma))
= -\epsilon(\sigma)$,
since only the wall-$w$ factor changes.

\medskip

\emph{Compatibility preserved.}
Since matching edges connect different walls,
every matching edge incident to~$2w{-}1$ or~$2w$
connects to a particle~$a$ from a different wall,
and $\sigma(a) \notin \{r, r+1\}$. Since positions~$r$
and~$r+1$ are adjacent integers:
\[
\sigma(a) > r
\iff \sigma(a) \geq r+1
\iff \sigma(a) > r+1.
\]
The second equivalence uses $\sigma(a) \neq r+1$.
Therefore swapping between~$r$ and~$r+1$ does not affect
any constraint.

\medskip

The involution~$\tau_w$ pairs the compatible
lifts with opposite signs, so the inner
sum vanishes for each~$M$, giving $c(I) = 0$.
\end{proof}

\subsection{Boundary vanishing}
\label{sec:boundary-vanishing}

\begin{theorem}[Boundary vanishing]
\label{thm:boundary-vanishing}
For $k \geq 2$, $c(I) = 0$ if wall~$1$ occupies
position~$1$, or if wall~$k$ occupies position~$2k$.
\end{theorem}

\begin{proof}
Suppose wall~$1$ occupies position~$1$, so either
$\sigma(1) = 1$ or $\sigma(2) = 1$ for every
lift~$\sigma$ of~$I$.

In any connected matching~$M$, vertex~$1$ has a matching
edge to some vertex~$v$ from a different wall, so
$v \geq 3$. Since $1 < v$, compatibility requires
$\sigma(1) > \sigma(v)$. But $\sigma(1) = 1$ is the
minimum of all positions, making this impossible.
If instead $\sigma(2) = 1$, the matching edge from
vertex~$2$ goes to some vertex~$v' \geq 3$
(matching edges connect different walls), with
$2 < v'$ and $\sigma(2) > \sigma(v')$ requiring
$1 > \sigma(v')$---again impossible.

\medskip

The argument for wall~$k$ at position~$2k$ is
symmetric: a particle assigned to position~$2k$
(the maximum) cannot satisfy $\sigma(v) > 2k$
for any matching edge from a different wall.
\end{proof}

\subsection{Indecomposable colorings}
\label{sec:indecomposability}

A coloring~$I = (I_1, \ldots, I_{2k})$ is
\emph{indecomposable} if there is no $0 < g < k$
such that $\{I_1, \ldots, I_{2g}\}
= \{1, \ldots, g\}$ (each appearing twice);
otherwise it is \emph{decomposable}.

\begin{theorem}[Indecomposability]
\label{thm:indecomposability}
If~$I$ is decomposable, then $c(I) = 0$.
\end{theorem}

\begin{proof}
Every connected matching~$M$ is a Hamiltonian cycle on
$\{1, \ldots, 2k\}$. Since all wall edges are
within-group (both endpoints in $\{1, \ldots, 2g\}$
or both in $\{2g{+}1, \ldots, 2k\}$), the matching must
contain at least one crossing edge $\{i, j\}$ with
$i \in \{1, \ldots, 2g\}$ and
$j \in \{2g{+}1, \ldots, 2k\}$ (otherwise the cycle would
disconnect into two components).

For this crossing edge with $i < j$, compatibility
requires $\sigma(i) > \sigma(j)$. But the coloring maps
the first group's particles to the first group's positions,
so $\sigma(i) \in \{1, \ldots, 2g\}$ and
$\sigma(j) \in \{2g{+}1, \ldots, 2k\}$. Therefore
$\sigma(i) \leq 2g < 2g{+}1 \leq \sigma(j)$, and the
constraint is never satisfied. No lift is
compatible with any connected matching, so $c(I) = 0$.
\end{proof}

\section{Variance and the central limit theorem}
\label{sec:clt}

For discrete skip-free processes at a fixed time
horizon, the wall indicators have finite-range
dependence, so a central limit theorem follows from
classical results for locally dependent variables.
The coloring formula gives more than the CLT
itself: it provides the cumulant bound
$\kappa_k(N_L) = O(L)$ for each $k \geq 2$
(\Cref{sec:clt-proof-discrete}), and immediately
shows that the covariance is negative (repulsion)
for every skip-free process
(\Cref{sec:covariance-coloring}).

\subsection{Central limit theorem, discrete case}
\label{sec:clt-proof-discrete}

Write $N_L = N([0,L])$ for the number of walls in
the interval~$[0,L]$.

\begin{theorem}[Cumulant bound: discrete]
\label{thm:clt}
Consider a coalescing process on~$\ZZ$ with
crossing-or-meeting probabilities satisfying
\begin{equation}\label{eq:tail-summability}
\sum_{d=1}^{\infty}\;
d^N \;
\sup_{a \in \ZZ}\;
\PP\bigl(X_a \geq X_{a+d}\bigr)
\;<\; \infty
\qquad \text{for every } N \geq 0,
\end{equation}
where $X_a, X_{a+d}$ are independent particles
started at~$a$ and~$a + d$.
Then for each fixed~$k \geq 2$,
\[
\kappa_k(N_L) = O(L)
\qquad \text{as } L \to \infty.
\]
\end{theorem}

Morally, condition~\eqref{eq:tail-summability}
says that particles far apart almost never cross,
and the crossing probability decays faster than
any polynomial in their separation.

\begin{proof}
Fix $k \geq 2$. The wall count
\[
N_L = \sum_{x \in \ZZhalf \cap [0,L]} \eta_x,
\]
where $\eta_x = \mathbf{1}\{\text{wall at } x\}$,
is a finite sum of $0$-$1$ variables.
By \Cref{prop:clt-indicators}, it suffices to show
that for each fixed~$k$,
\begin{equation}\label{eq:per-site}
\sup_{x_0 \in \ZZhalf \cap [0,L]}\;
  \sum_{\substack{A \ni x_0,\; |A| = k \\
  A \subseteq \ZZhalf \cap [0,L]}}
  |\kappa_A| = O(1).
\end{equation}

\medskip
\noindent\textbf{Step 1: Reduction to indecomposable
colorings.}
Fix any $x_0 \in \ZZhalf$ and a set~$A$ of
size~$k$ containing~$x_0$. The coloring formula
(\Cref{thm:coloring-formula}) gives
\[
|\kappa_A|
\;\leq\;
\sum_I |c(I)| \cdot
  \PP_A\bigl(\text{coloring} = I,\;
  \text{all non-crossing}\bigr),
\]
where $\PP_A$ is over $2k$ independent particles
(two flanking each wall).
By \Cref{thm:indecomposability}, $c(I) = 0$
unless~$I$ is indecomposable. Restricting to
indecomposable~$I$ and exchanging the sums
over~$A$ and~$I$:
\begin{equation}\label{eq:DI-bound}
\text{LHS of~\eqref{eq:per-site}}
\;\leq\;
\sum_{\substack{I \text{ indecomp.}}} |c(I)|
  \sum_{\substack{A \ni x_0 \\ |A|=k}}
  \PP_A\bigl(\text{coloring} = I,\;
  \text{all non-crossing}\bigr).
\end{equation}

\medskip
\noindent\textbf{Step 2: Crossing at the widest gap.}
We show that the inner sum is finite. Write
$A = \{x_1 < \cdots < x_k\}$ with
$x_0 = x_j$ for some~$j$, and let
$d_{\max} = \max_g (x_{g+1} - x_g)$ be the
widest gap between consecutive walls.

Consider the split at this widest gap:
walls~$x_1, \ldots, x_{g^*}$ on the left,
$x_{g^*+1}, \ldots, x_k$ on the right.
Since~$I$ is indecomposable, the $2k$ final
positions (sorted left to right) do not
decompose at~$g^*$: some particle from the
left group ends up to the right of some
particle from the right group. These two
particles have crossed or met. By the union
bound over the at most~$4k^2$ pairs of
flanking particles,
\[
\PP_A(\cdots)
\;\leq\;
4k^2 \cdot \max_{\substack{a \leq x_{g^*}
  + \frac{1}{2} \\
  b \geq x_{g^*+1} - \frac{1}{2}}}
  \PP\bigl(X_a \geq X_b\bigr).
\]

\medskip
\noindent\textbf{Step 3: Summation over~$A$.}
Since~$b - a \geq d_{\max} - 1$ for any such
pair, and~$d_{\max} \geq
\max_i |x_i - x_0| / (k{-}1)$, summing
over all sets~$A$ with $d_{\max} = d$
(at most $(2(k{-}1)d{+}1)^{k-1}$ choices)
gives
\[
\sum_{\substack{A \ni x_0 \\ |A|=k}}
  \PP_A(\cdots)
\;\leq\;
4k^2 \sum_{d=1}^{\infty}
  (2(k{-}1)d{+}1)^{k-1} \cdot
  \sup_{a}\;
  \PP\bigl(X_a \geq X_{a+d-1}\bigr),
\]
which converges
by~\eqref{eq:tail-summability}
with $N = k - 1$.
Since the outer sum in~\eqref{eq:DI-bound}
has finitely many indecomposable colorings,
\eqref{eq:per-site} holds.
\end{proof}

\begin{remark}[Local dependence and scaling limits]
\label{rem:local-dependence}
For a discrete-time skip-free process
with~$T$ steps, the wall process has finite-range
dependence (range~$2T$), and the CLT is classical.
The per-site bound depends on the time horizon:
when~$T$ grows with the window~$L$, the crude bound
is useless. For example, if $L \sim \sqrt{T}$, the
dependence range $2T \sim L^2$ grows faster than the
window. This is where indecomposability
(\Cref{thm:indecomposability}) becomes essential: it
ensures that the per-site sum remains bounded as
$T \to \infty$, by localizing all simplex variables
through gap coverage
(\Cref{sec:non-colliding-conditioning}).
\end{remark}

\subsection{Central limit theorem, continuous case}
\label{sec:clt-proof-continuous}

For continuous processes, the tail
condition~\eqref{eq:tail-summability} cannot be
applied directly to a refined grid: as the grid
mesh~$\epsilon \to 0$, each wall's flanking
particles start at distance~$\epsilon$, so
$\PP(X_a \geq X_{a+\epsilon}) \to 1/2$ and the
sum diverges. The non-colliding conditioning
(\Cref{sec:non-colliding-conditioning}) resolves
this by factoring out the vanishing prefactor.

\begin{theorem}[Cumulant bound: continuous]
\label{thm:clt-continuous}
Consider a coalescing process on~$\RR$ with
continuous paths under the maximal entrance law.
For flanking particles at distance~$\epsilon$,
condition on non-crossing during~$[0, T]$
(\Cref{sec:non-colliding-conditioning}).
Assume:
\begin{enumerate}[label=\textup{(\roman*)},
  ref=\textup{(\roman*)}]
\item \label{item:wall-intensity}
  \emph{Wall intensity.} The limit
  \begin{equation}\label{eq:wall-intensity}
  \varrho(x)
  = \lim_{\epsilon \to 0}
    \frac{\PP(\textup{no crossing at
      distance } \epsilon \textup{ near } x)}
      {\epsilon}
  \end{equation}
  exists (the wall intensity may depend on the
  spatial position~$x$). In particular,
  $\EE[N_L] = \int_0^L \varrho(x)\, dx$, so
  the $k = 1$ bound is immediate.
  
\item \label{item:limiting-cond}
  \emph{Limiting conditioned process.}
  As $\epsilon \to 0$, the conditioned
  non-colliding pair converges to a
  limiting process (the Doob $h$-transform
  with $h(a,b) = b - a$).
\item \label{item:tail-cond-cont}
  \emph{Tail condition.} The
  crossing-or-meeting probabilities of
  particles from different limiting
  conditioned pairs satisfy
  \begin{equation}\label{eq:tail-summability-cont}
  \int_0^{\infty}
  \delta^N \;
  \sup_{a}\;
  \PP_{\mathrm{cond}}\bigl(
    Y_a^{(1)} \geq Y_b^{(2)}\bigr)
  \,\mathrm{d} \delta
  \;<\; \infty
  \qquad \text{for every } N \geq 0,
  \end{equation}
  where $Y_a^{(1)}, Y_b^{(2)}$ are final
  positions of particles from two different
  limiting conditioned pairs starting at
  distance~$\geq \delta$.
\end{enumerate}

Then for each fixed~$k \geq 2$,
\[
\kappa_k(N_L) = O(L)
\qquad \text{as } L \to \infty.
\]
\end{theorem}

\begin{proof}
Detect walls on a grid of mesh~$\epsilon$:
partition~$[0,L]$ into intervals of
length~$\epsilon$ and let
$\eta_j^{(\epsilon)} = \mathbf{1}\{\text{wall
in the $j$-th interval}\}$,
$N_L^{(\epsilon)} = \sum_j \eta_j^{(\epsilon)}$.
Each~$\eta_j^{(\epsilon)}$ is a $0$-$1$ indicator,
and the coloring formula
(\Cref{thm:coloring-formula}) applies with flanking
particles at the interval endpoints.

By the non-colliding conditioning
(\Cref{sec:non-colliding-conditioning}), each
cumulant factors as
\[
\kappa_A
= \prod_{w} P_w(\text{no crossing})
  \cdot \EE_{\mathrm{cond}}\bigl[c(I)\bigr],
\]
so
\[
|\kappa_A|
\;\leq\;
\prod_{w} P_w(\text{no crossing})
\cdot
\sum_I |c(I)| \cdot
  \PP_{\mathrm{cond}}\bigl(
  \text{coloring} = I\bigr).
\]
Apply the discrete-case argument
(\Cref{thm:clt}) to the conditional
probability: indecomposability
(\Cref{thm:indecomposability}) and the union
bound control the per-site sum.
By~\eqref{eq:wall-intensity},
$P_w(\text{no crossing})
= \varrho(x_w)\,\epsilon + o(\epsilon)$
for each wall at position~$x_w$, contributing a
factor
$\prod_w \varrho(x_w)\,\epsilon$ to $|\kappa_A|$.
The sums
over the $k - 1$ free wall positions on the
$\epsilon$-grid are Riemann sums: each sum
contributes a factor~$\epsilon$ times an
integral, and the integrals converge by the
conditioned tail
condition~\ref{item:tail-cond-cont}.
By~\ref{item:limiting-cond}, the conditioned
process converges as $\epsilon \to 0$, so the
Riemann sums converge to the corresponding
integrals.
Altogether,
$\kappa_k(N_L^{(\epsilon)}) = O(L)$
uniformly in~$\epsilon$.

As $\epsilon \to 0$, each interval contains at
most one wall (the expected wall count
$\int_0^L \varrho(x)\, dx$ is finite
by~\ref{item:wall-intensity}, so at most finitely
many walls exist almost surely), and
$N_L^{(\epsilon)} \to N_L$. Since
$0 \leq N_L^{(\epsilon)} \leq N_L$ and $N_L$ has
finite moments (it is bounded by the number of
walls in~$[0,L]$, which is almost surely finite
by~\ref{item:wall-intensity}), dominated convergence
gives convergence of all moments, hence
$\kappa_k(N_L) = O(L)$.
\end{proof}

\begin{remark}\label{rem:clt-bm}
For coalescing Brownian motion, the conditioned
pairs are non-colliding Brownian motions (the Doob
$h$-transform with $h(a,b) = b - a$). The
conditioned final positions have densities involving
$\varphi(y - x)$ (where $\varphi$ is the standard
normal density) and
$(y - x)\varphi(y - x)$
(\Cref{sec:non-colliding-conditioning}), so the
crossing-or-meeting probability between different
conditioned pairs decays with Gaussian tails,
and~\eqref{eq:tail-summability-cont} holds.
\end{remark}

For coalescing Brownian motion, the $k = 2$
case of \Cref{thm:clt-continuous} gives
$\kappa_2(N_L) = O(L)$; the exact variance
constant~$(\sqrt{2}-1)^2$ per unit length was
computed by Glinyanaya and
Fomichov~\cite{GlinyanayaFomichov2018}, giving
$\kappa_2(N_L) = \Theta(L)$.
The standard CLT criterion
(see, e.g.,~\cite{Soshnikov2002}) then gives:

\begin{corollary}[CLT for the wall count]
\label{cor:clt}
For coalescing Brownian motion under the maximal
entrance law,
\[
\frac{N_L - \EE[N_L]}
     {\sqrt{\operatorname{Var}(N_L)}}
\;\xrightarrow{d}\; \mathcal{N}(0,1)
\qquad \text{as } L \to \infty.
\]
\end{corollary}

\subsection{Discussion}
\label{sec:clt-discussion}

\subsubsection{Structural insight from indecomposability}

The coloring formula adds structural insight beyond
the classical CLT: indecomposability explains
\emph{why} the dependence is local (every gap
between wall positions must be straddled by some
particle pair, preventing long-range correlations),
and the $k = 2$ analysis gives the sign of the
covariance for free
(\Cref{sec:covariance-coloring}).

\subsubsection{Comparison with kernel-based CLTs}

For time-homogeneous processes, a CLT can also be
obtained from the explicit Pfaffian
kernel~\cite{GlinyanayaFomichov2018,WangXu2025,LinQiuWang2021}.
The coloring formula provides a different route
that bypasses the kernel entirely, and therefore
applies to inhomogeneous processes where no kernel
is available. Moreover, the wall-level CLT
(\Cref{thm:clt}) is entirely self-contained: it
uses only the coloring formula and the
crossing-or-meeting probabilities of the original
process, with no duality or lattice structure
needed. The checkerboard enters only when
translating the wall CLT into a CLT for the
surviving particles
(\Cref{sec:walls-to-particles}).

\section{From empty intervals to point processes}
\label{sec:empty-to-pp}

The Pfaffian empty-interval formula
(\Cref{thm:pfaffian-empty-interval}) gives the
probability that specified intervals contain no
wall as a Pfaffian whose entries are
crossing-or-meeting probabilities of independent
particles. This section shows how to recover
correlation functions---the probability of finding
walls at specified sites---from the empty-interval
formula, yielding a Pfaffian point process for
the wall configuration.

\subsection{Continuous setting}
\label{sec:empty-to-pp-continuous}

In the continuous (Brownian motion) setting,
differentiation suffices: the $n$-point correlation
function is obtained by differentiating the
$n$-interval empty probability with respect to
the interval lengths and taking the zero-length
limit. For coalescing Brownian motions under the
maximal entrance law, this recovers the
Tribe--Zaboronski Pfaffian point
process~\cite{TZ2011,GarrodPTZ2018};
\Cref{sec:pfaffian} records the explicit kernel.

\subsection{Discrete setting}
\label{sec:empty-to-pp-discrete}

In the discrete setting, M\"obius inversion on the
Boolean lattice converts empty-interval probabilities
into correlation functions. Garrod, Poplavskyi, Tribe,
and Zaboronski~\cite[Theorem~1]{GarrodPTZ2018} showed
that the alternating sum of Pfaffians collapses to a
single Pfaffian with a $2 \times 2$ matrix kernel:
the passage from empty to occupied indicators acts as
a forward difference on the Pfaffian entries.
Feeding our empty-interval formula
(\Cref{thm:pfaffian-empty-interval}) into their
reconstruction gives the following.

\begin{proposition}[Discrete Pfaffian point process
{\cite[Theorem~1]{GarrodPTZ2018}}]
\label{prop:discrete-pp}
Under the maximal entrance law for any discrete
skip-free coalescing process with arbitrary
inhomogeneous transition probabilities,
the wall configuration at time~$T$
is a Pfaffian point process: the $n$-point
correlation function at
sites~$y_1 < \cdots < y_n$ on~$\ZZhalf$ is
\[
\rho_n(y_1, \ldots, y_n)
= \Pf\bigl[K(y_i, y_j)\bigr]_{1 \leq i,j \leq n},
\]
where $K$ is the $2 \times 2$ antisymmetric matrix
kernel
\begin{equation}
	\label{eq:discrete-kernel}
	K(y,z) = \begin{pmatrix}
		C(y,z) &
		-\Delta_z\, C(y,z) \\[0.3em]
		-\Delta_y\, C(y,z) &
		\Delta_y \Delta_z\, C(y,z)
	\end{pmatrix}
\end{equation}
with $C(y,z) \coloneqq
A(y {-} \tfrac{1}{2},\, z {-} \tfrac{1}{2})$ the
crossing-or-meeting probability of
\Cref{thm:pfaffian-empty-interval} evaluated at
positions $y {-} \tfrac{1}{2}$
and~$z {-} \tfrac{1}{2}$, and
$\Delta_y f(y) \coloneqq f(y{+}1) - f(y)$ the
forward difference.
\end{proposition}

\section{The Brownian motion case}
\label{sec:direct-examples}

The Pfaffian empty-interval formula
(\Cref{thm:pfaffian-empty-interval}) applies to any
skip-free coalescing system. This section works out
the crossing-or-meeting probabilities for Brownian
motion on~$\RR$---a setting where these
probabilities have an explicit closed form. No
lattice duality or checkerboard construction is
needed; the formula operates directly on the
process.

\label{sec:bm-example}
\label{sec:pfaffian}
\label{ex:bm}

\subsection{The crossing-or-meeting probability}

Consider coalescing Brownian motions on~$\RR$
under the maximal entrance law: every point of~$\RR$
is initially occupied. The paths are continuous,
so the process is skip-free. Fix a time~$T > 0$.
Walls are real-valued: each wall marks a boundary
between two adjacent basins of attraction.

For two independent Brownian motions starting
at~$a < b$, the difference
$D(s) = B_b(s) - B_a(s)$ is a Brownian motion
starting at~$b - a > 0$ with variance~$2s$. The
particles cross or meet when $D$ hits zero. By the
reflection principle,
\begin{equation}
\label{eq:bm-crossing}
\PP(\text{crossing or meeting by time~$T$})
  = 2\,\Phi\!\left(
  -\frac{b-a}{\sqrt{2T}}\right)
  = \erfc\!\left(\frac{b-a}{2\sqrt{T}}\right),
\end{equation}
where $\Phi$ is the standard normal distribution
function.

\subsection{The Pfaffian formula}

Applying \Cref{thm:pfaffian-empty-interval}, the
probability that no wall lies in any of the
intervals $(a_1, b_1), \ldots, (a_n, b_n)$ is
$\Pf(A)$, where the $2n \times 2n$ antisymmetric
matrix~$A$ has entries
\[
A_{kl}
  = \erfc\!\left(
    \frac{|x_l - x_k|}{2\sqrt{T}}\right),
    \qquad k < l,
\]
and $x_1, \ldots, x_{2n}$ are the $2n$ endpoints
in order.

For a single interval~$(a, b)$, this gives
\[
\PP(\text{no wall in $(a, b)$})
  = \erfc\!\left(\frac{b - a}{2\sqrt{T}}\right),
\]
recovering the empty-interval probability of
Doering and
ben-Avraham~\cite{DoeringBenAvraham1988}. For
multiple intervals, the Pfaffian of erfc entries
recovers the formulas of
ben-Avraham~\cite{benAvraham1998}.
ben-Avraham and
Brunet~\cite{benAvrahamBrunet2005} used the same
empty-interval method to derive consecutive-particle
distributions and the thinning relation between
coalescence and annihilation correlations.

\subsection{The Tribe--Zaboronski kernel}

In rescaled coordinates (measuring positions in
units of~$\sqrt{T}$), the crossing-or-meeting
probability for positions~$x$ and~$y$ is
$\mathcal{F}(|y - x|)$, where
\[
\mathcal{F}(z)
  = \erfc(z/2)
  = \frac{1}{\sqrt{\pi}}
    \int_z^\infty e^{-u^2/4}\, du.
\]
Differentiating the Pfaffian empty-interval formula
with respect to the interval lengths and taking the
zero-length limit (\Cref{sec:empty-to-pp-continuous})
yields the $n$-point correlation
function as a Pfaffian with a $2 \times 2$ matrix
kernel:
\begin{equation}
\label{eq:tz-kernel}
K(x, y) = \begin{pmatrix}
-\mathcal{F}''(y-x) &
  -\mathcal{F}'(y-x) \\[0.3em]
\mathcal{F}'(y-x) &
  \sgn(y-x)\, \mathcal{F}(|y-x|)
\end{pmatrix}.
\end{equation}
This recovers the Tribe--Zaboronski Pfaffian point
process~\cite{TZ2011,GarrodPTZ2018}.

\begin{remark}
\label{rem:inhomogeneous-bm}
For Brownian motion with position- and
time-dependent diffusion coefficient and drift,
the empty-interval probabilities are still given
by a Pfaffian (the crossing-or-meeting
probabilities depend on the variable
coefficients). To extract a $2 \times 2$ matrix
kernel requires differentiating with respect to
the interval lengths, which imposes regularity
conditions on the variable coefficients. This
analytical problem is beyond the scope of this
paper.
\end{remark}

Coalescing Brownian motions satisfy the assumptions
of \Cref{thm:clt-continuous}; see \Cref{rem:clt-bm}.

\section{From walls to particles: checkerboard duality}
\label{sec:checkerboard}

This section constructs the checkerboard duality that
translates the wall-level results of
\Cref{sec:walls-pfaffian,sec:cumulants,sec:clt} into
statements about surviving particles.

\subsection{The dual forests}
\label{sec:checkerboard-decomposition}

\subsubsection{The lattice and the voter model}

Consider the integer lattice~$\ZZ^2$ and the
half-integer lattice
$\ZZhalf^2 = (\ZZ + \tfrac{1}{2})^2$, jointly
embedded in the $(u,v)$ plane. On each
diagonal~$u + v = n$ (where $n$ is an integer), the
two lattices alternate---$\ZZ^2$-vertices at integer
$u$-coordinates and $\ZZhalf^2$-vertices at
half-integer $u$-coordinates---like black and white
squares on a checkerboard (rotating
the $(u,v)$ plane by $45^\circ$ makes the two
sublattices visually checkerboard-like).
In the voter model
interpretation~\cite{HolleyLiggett1975},
$\ZZhalf^2$-vertices hold \emph{opinions} and
$\ZZ^2$-vertices carry \emph{boundaries} where
neighboring opinions differ.

\subsubsection{The backward opinion forest}

Each $\ZZhalf^2$-vertex
$(u + \tfrac{1}{2}, v + \tfrac{1}{2})$ copies its
opinion from one of two $\ZZhalf^2$-neighbors on
the previous diagonal:
\begin{itemize}
\item from the West neighbor
  $(u - \tfrac{1}{2}, v + \tfrac{1}{2})$ with
  probability~$p_{u,v}$;
\item from the South neighbor
  $(u + \tfrac{1}{2}, v - \tfrac{1}{2})$ with
  probability~$1 - p_{u,v}$;
\end{itemize}
where the weights $p_{u,v} \in [0,1]$ may vary
from vertex to vertex. The backward
edge connects each vertex to the neighbor whose
opinion it copied. Following backward edges traces
the chain of opinion inheritance to a common
ancestor. The selected backward edges form the
\emph{backward opinion forest} on~$\ZZhalf^2$,
with edges going West or South in the direction of
decreasing~$u + v$.

\subsubsection{The forward boundary forest}

On diagonal~$u + v = n$, a boundary exists at
$(u, n-u) \in \ZZ^2$ when the opinions at the
adjacent $\ZZhalf^2$-vertices
$(u - \tfrac{1}{2}, n - u + \tfrac{1}{2})$ and
$(u + \tfrac{1}{2}, n - u - \tfrac{1}{2})$ disagree.
Each backward choice at a $\ZZhalf^2$-vertex
determines a forward edge at the adjacent
$\ZZ^2$-vertex~$(u,v)$: copying from the West
produces a forward East edge (from $(u,v)$ to
$(u{+}1,v)$); copying from the South produces a
forward North edge (from $(u,v)$ to $(u,v{+}1)$).
These forward edges form the complementary
\emph{forward boundary forest} on~$\ZZ^2$, with
edges in the direction of increasing~$u + v$. The
two forests are non-crossing: the forward and
backward edges at each vertex do not cross in the
planar embedding (East pairs with West, North with
South; the crossing pairs East--South and
North--West are excluded).

Harris~\cite{Harris1978} introduced the graphical
construction: a single collection of independent
random choices that couples all initial
configurations on the same probability space.
Arratia~\cite{Arratia1979} observed that the same
random choices produce two complementary
non-crossing forests on interleaved sublattices (his
``percolation substructure''); the forward forest
carries boundaries and the backward forest carries
opinions.

\subsection{Boundaries as coalescing walks}
\label{sec:checkerboard-duality}

\subsubsection{Boundary propagation and coalescence}

Since all initial opinions are distinct (the maximal
entrance law), every $\ZZ^2$-vertex on
diagonal~$u + v = 0$ carries a boundary. Each
boundary propagates forward along the forward
forest: if a boundary exists at~$(u,v)$ and the
forward edge goes East, then the same boundary
appears at~$(u+1,v)$ on the next diagonal (because
the backward edge copies the opinion from the same
side, preserving the disagreement).
When two boundary paths meet at the same
vertex, they coalesce into one. The surviving
boundaries on~$\ZZ^2$ form coalescing walks---these
are the coalescing random walks of
Arratia~\cite{Arratia1979}. Symmetrically, starting
with opinions on diagonal~$u + v = T$
on~$\ZZhalf^2$ and following the backward edges
yields coalescing opinion
lineages on~$\ZZhalf^2$.

At a fixed diagonal~$u + v = T$, the coalescing
boundaries are precisely the walls of the backward
process (\Cref{sec:walls-coalescence}): each
boundary separates two adjacent opinion basins, and
the absence of a boundary between two vertices means
their backward lineages have coalesced.

\subsubsection{Time as a coordinate choice}

Time is derived, not primitive:
setting $t = u + v$ foliates the lattice into time
slices, and different space functions $x = g(u,v)$
yield different particle dynamics
(\Cref{sec:examples}).

\subsubsection{From walls to particles}
\label{sec:walls-to-particles}

To study the surviving particles at time~$T$, apply
the wall-level results of
\Cref{sec:walls-pfaffian,sec:cumulants,sec:clt} to
the dual backward process, whose walls are precisely
the surviving boundaries of the forward process.

\subsection{Connection with the spin-pair duality}
\label{sec:gptz-connection}

\subsubsection{The dual particles}

The spin-pair duality
of~\cite{TZ2011,GarrodPTZ2018} arrives at the same
mathematical object by a different route. Their
scalar kernel
$K_t(y,z) = \EE[\sigma_{y,z}(\eta_t)]$,
where
$\sigma_{y,z}(\eta) = (-\theta)^{\eta[y,z)}$
is the spin-pair function,
satisfies the PDE
$(L_y + L_z)\, K_t = \partial_t K_t$,
where $L$ is the generator of a single particle.
Garrod, Poplavskyi, Tribe, and
Zaboronski~\cite{GarrodPTZ2018} describe $L$ as
``the generator for a single dual particle.'' These
dual particles are our backward particles: the PDE
says that two independent particles with
generator~$L$ evolve at positions $y$ and~$z$,
which is precisely what the Pfaffian formula
computes---crossing-or-meeting probabilities
of independent backward particles from the
interval endpoints.

\subsubsection{Boundary condition and reconstruction}

The boundary condition $K_t(y,y) = 1$ is automatic
in the checkerboard framework: two particles
starting at the same position have
crossing-or-meeting probability~$1$.
The reconstruction step---passing from the scalar
kernel~$K_t$ to the $2 \times 2$ matrix kernel via
discrete differences
(\Cref{prop:discrete-pp})---also coincides:
in~\cite{GarrodPTZ2018}, the identity
$\eta(y) = (1 - \sigma_{y,y+1})/(1+\theta)$
produces the same forward differences~$\Delta_y$
applied to the same scalar kernel.

\subsubsection{Algebraic vs.\ geometric discovery}

The two approaches discover the dual particles by
different means: the spin-pair duality finds them
algebraically (the generator~$L$ emerges from
computing the action on spin pairs), while the
checkerboard construction makes them geometrically
visible as paths on the lattice. For
time-homogeneous dynamics with $\theta = 0$, the
resulting Pfaffian point processes are identical.

\subsubsection{Non-maximal initial conditions}

The correspondence extends to non-maximal initial
conditions. In~\cite{GarrodPTZ2018}, the initial
condition~$\eta$ enters through the PDE initial
data $K_0(y,z) = \sigma_{y,z}(\eta)$; for
$\theta = 0$, this is
$K_0(y,z) = \mathbf{1}\{\eta[y,z) = 0\}$, the
indicator that the interval~$[y,z)$ is initially
empty. In the checkerboard framework, the
analogous role is played by the clamping step
(\Cref{sec:general-initial}). Both encode the
same data---which intervals are initially
empty---by different means.

\subsection{General initial conditions}
\label{sec:general-initial}

\subsubsection{Two meanings of ``no wall''}

The wall-level results
(\Cref{sec:walls-pfaffian,sec:cumulants,sec:clt})
assume the maximal entrance law: every site is
initially occupied. In applications one often starts
from a non-maximal initial condition---for instance,
a finite configuration. The proof of the Pfaffian
formula (\Cref{thm:pfaffian-empty-interval}) relies
on the equivalence: ``no wall in~$(a,b)$'' equals
``the flanking particles from~$a$ and~$b$ have
coalesced''
(\Cref{prop:no-wall-coalescence}). Under a general
initial condition this equivalence breaks: ``no wall
in~$(a,b)$'' can also hold when the flanking
particles have not coalesced but no initial particle
exists between~$a$ and~$b$.

\subsubsection{The clamping step}

To restore the equivalence, adjoin one deterministic
step to the dual backward process (from time~$0$
to time~$-1$): clamp together all backward particles
that have no initial forward particle between them.
After clamping, ``no wall in~$(a,b)$'' once again
implies that the (augmented) backward particles have
coalesced---either during the $T$~random steps
(from time~$T$ to time~$0$) or by the clamp.
The clamping step is automatically
skip-free: a particle cannot jump over another when
there is nothing between them to jump over.

The augmented process---$T$~random steps and one
deterministic---is an inhomogeneous
skip-free process, and the Pfaffian
formula~\cite{Sniady2026annihilation} allows
arbitrary inhomogeneous transition probabilities,
including deterministic ones.

\subsection{The augmented kernel}
\label{sec:general-initial-extension}

\begin{theorem}[Pfaffian formula for general
initial conditions]
\label{thm:pfaffian-general-ic}
Consider a coalescing skip-free process with
deterministic initial condition~$\eta$ (a subset
of occupied sites). Run the system for time~$T$,
and let $a_1 < b_1 \leq a_2 < \cdots \leq a_n
< b_n$ be given. Then:
\[
\PP(\text{no wall in any $(a_i, b_i)$})
  = \Pf(A),
\]
where $A$ is the $2n \times 2n$ antisymmetric
matrix whose entry~$A_{kl}$ ($k < l$) is the
probability that independent backward particles
from~$x_k$ and~$x_l$ either (a) cross or meet during
the $T$~steps, or (b) do not cross or meet but $\eta$
has no particle between their final positions.
\end{theorem}

Under the maximal entrance law the second
alternative is vacuous, recovering
\Cref{thm:pfaffian-empty-interval}.

\begin{proof}
The clamping step restores the equivalence between
``no wall'' and ``pairwise coalescence'' for the
augmented process
(\Cref{prop:no-wall-coalescence}). The augmented
process---$T$~random steps and one
deterministic---is an inhomogeneous skip-free
process, so the Pfaffian annihilation
formula~\cite{Sniady2026annihilation} applies.
The crossing-or-meeting probability for the
augmented process is exactly the entry~$A_{kl}$
described above.
\end{proof}

Since the augmented process is skip-free and
inhomogeneous, the cumulant coloring formula
(\Cref{thm:coloring-formula}) and the central
limit theorem (\Cref{thm:clt}) also hold for the
augmented kernel, extending both results to
general initial conditions.

\section{Checkerboard examples}
\label{sec:examples}

The checkerboard lattice carries no intrinsic notion
of space and time. A choice of time function
$t = f(u,v)$ and space function $x = g(u,v)$ converts
the forests into particle trajectories: on each level
set $t = \text{const}$, the vertices become particle
positions and forward-forest paths become
trajectories. This section works out explicit formulas
for several particle systems arising from different
coordinate choices.

\subsection{Biased random walk}
\label{sec:biased-rw}

Fix a probability~$p \in (0,1)$ and set $q = 1 - p$.
Set all edge weights equal: $p_{u,v} = p$ for every
vertex. In Arratia's coordinates $t = u + v$,
$x = u - v$, the forward boundary forest on~$\ZZ^2$
gives coalescing random walks on~$\ZZ$: at each time
step, a boundary particle jumps $+1$ (East in the
$(u,v)$ plane) with probability~$p$, or $-1$ (North)
with probability~$q$.

\subsubsection{Backward dynamics}

The Pfaffian formula
(\Cref{thm:pfaffian-empty-interval}) involves \emph{backward}
particles on~$\ZZhalf^2$. Each backward particle
follows the edges of the backward forest---these are
random edges, determined by the same random choices
that generate the forward forest
(\Cref{sec:checkerboard-decomposition}). The
backward particle does not make its own random
choices; it traces the pre-existing backward arrows.

When the random choice at~$(u,v) \in \ZZ^2$ selects
East (probability~$p$), the complementary backward
edge goes West. In $(t,x)$-coordinates, West means
$x \to x - 1$. When the choice selects North
(probability~$q$), the backward edge goes South,
meaning $x \to x + 1$. Since the edge weights are
constant, the backward steps are independent and
identically distributed, giving a random walk with
\emph{swapped probabilities}: the backward particle
jumps $-1$ with probability~$p$ and $+1$ with
probability~$q$.

\begin{remark}
\label{rem:backward-reversal}
When $p \neq q$, the backward walk differs from the
forward walk: forward boundaries jump right with
probability~$p$, while backward particles jump right
with probability~$q$. The reversal arises because
the forward and backward edges at each vertex point in
opposite directions. For $p = q = \tfrac{1}{2}$
(\Cref{sec:srw-example}), the reversal is trivial
and both walks are simple symmetric random walks.

For inhomogeneous edge weights~$p_{u,v}$, the
backward particle follows a path through a random
environment where the step probabilities vary from
vertex to vertex, determined by the specific
weights~$p_{u,v}$ encountered along the backward
path. The transition weight~\eqref{eq:bwd-transition}
must then be computed for the inhomogeneous
probabilities.
\end{remark}

\subsubsection{The crossing-or-meeting probability}

The backward transition weight---the probability
that a backward particle at position~$x$ on
diagonal~$u + v = T$ reaches position~$y$ on
diagonal~$u + v = 0$ after $T$~backward steps---is
\begin{equation}
\label{eq:bwd-transition}
w_T(x, y) \;=\;
  \binom{T}{\frac{T + y - x}{2}}\,
  q^{\frac{T + y - x}{2}}\,
  p^{\frac{T - y + x}{2}},
\end{equation}
where $\tfrac{T + y - x}{2}$ counts the rightward
jumps (each with probability~$q$ in the backward
walk), and the expression is zero
unless $\tfrac{T + y - x}{2}$ is a non-negative
integer at most~$T$.

For two backward particles at $x$-positions
$x_I < x_J$ on diagonal~$u + v = T$, the probability that they cross or meet
(\Cref{sec:cancellative-labeling};
see~\cite{Sniady2026annihilation}) is
\begin{equation}
\label{eq:Aij-general}
A_{IJ} \;=\; 2\!\!\sum_{y_1 < y_2}
  w_T(x_I, y_2)\, w_T(x_J, y_1)
  \;\;+\; \sum_{y} w_T(x_I, y)\, w_T(x_J, y),
\end{equation}
where $y_1, y_2$ range over all positions reachable
in $T$ backward steps.
In each term, particle~$I$ (starting at the
smaller position~$x_I$) ends strictly right of
particle~$J$: the paths have crossed.
The factor~$2$ arises because expanding
$A_{IJ} = 1 - \det(\cdots)$ yields the crossing
probability twice (once from the total, once
from the sign reversal in the determinant).
The second sum counts pairs ending at the same
position. Equivalently,
$A_{IJ}$ equals the total probability minus the
Karlin--McGregor non-crossing
probability~\cite{KM1959}:
\[
A_{IJ} \;=\; 1 \;-\; \sum_{y_1 < y_2}
  \det \begin{pmatrix}
    w_T(x_I, y_1) & w_T(x_I, y_2) \\
    w_T(x_J, y_1) & w_T(x_J, y_2)
  \end{pmatrix}.
\]
We set $A_{JI} = -A_{IJ}$ to obtain an
antisymmetric matrix.

\subsubsection{The Pfaffian empty-interval formula}

In $(t,x)$-coordinates,
\Cref{thm:pfaffian-empty-interval} reads as follows.
For half-integer positions $a < b$, the
$2 \times 2$ Pfaffian
reduces to a single matrix entry:
\[
\PP(\text{no wall in $(a, b)$ at time $T$}) = A_{12}
\quad \text{with } x_1 = a,\; x_2 = b.
\]

For two disjoint intervals $(a_1, b_1)$ and
$(a_2, b_2)$ with $a_1 < b_1 < a_2 < b_2$, the
$4 \times 4$ Pfaffian expands as
\[
\Pf(A) = A_{12}\, A_{34}
       - A_{13}\, A_{24}
       + A_{14}\, A_{23},
\]
where each $A_{IJ}$ is computed
from~\eqref{eq:Aij-general}. Every entry is an
explicit double sum of products of binomial
coefficients.

\subsection{Symmetric random walk}
\label{sec:srw-example}
\label{ex:symmetric}

When $p = q = \tfrac{1}{2}$, the forward and
backward dynamics coincide: both are simple symmetric
$\pm 1$ random walks. This is the original setting
of Arratia~\cite{Arratia1979}.

\subsubsection{The crossing-or-meeting probability}

The transition weight (probability of moving
from~$x$ to~$y$ in $T$~steps) is
\begin{equation}
\label{eq:srw-transition}
w_T(x, y) = \binom{T}{\frac{T + y - x}{2}}\,
  \frac{1}{2^T},
\end{equation}
where the expression is zero unless
$\tfrac{T + y - x}{2}$ is a non-negative integer
at most~$T$.

For two independent particles starting
at~$a < b$, the crossing-or-meeting probability
equals~$1$ minus the
Karlin--McGregor~\cite{KM1959} non-crossing
probability:
\begin{equation}
\label{eq:srw-crossing}
A_{ab} = 1 - \sum_{y_1 < y_2}
  \det \begin{pmatrix}
    w_T(a, y_1) & w_T(a, y_2) \\
    w_T(b, y_1) & w_T(b, y_2)
  \end{pmatrix}.
\end{equation}

\subsubsection{The Pfaffian formula}

By \Cref{thm:pfaffian-empty-interval}, the
probability that no wall lies in any of the
intervals $(a_1, b_1), \ldots, (a_n, b_n)$ is
$\Pf(A)$, where $A$ is the $2n \times 2n$
antisymmetric matrix with
entries~\eqref{eq:srw-crossing}. Every entry is
an explicit finite sum of products of binomial
coefficients.

Under diffusive scaling, the transition
weight~$w_T$ converges to the Gaussian density
and the crossing-or-meeting
probability~\eqref{eq:srw-crossing} converges to
the complementary error
function~\eqref{eq:bm-crossing}, recovering the
Brownian motion case.

\subsection{Totally asymmetric dynamics}
\label{sec:asymmetric-walk}
\label{ex:asymmetric}

In coordinates $t = u + v$, $x = u$, a forward
boundary particle jumps $+1$ (East) with
probability~$p$ or stays (North) with
probability~$q$. The backward particle jumps $-1$
or stays, giving the transition weight
\[
w_T(x, y) = \binom{T}{x - y}\,
  p^{x-y}\, q^{T - x + y}
\]
for $0 \leq x - y \leq T$. Substituting into the
Karlin--McGregor complement gives $A_{IJ}$ as a
finite double sum of products of binomial
coefficients, and
\Cref{thm:pfaffian-empty-interval} yields a Pfaffian
empty-interval formula for this totally asymmetric
system. All existing Pfaffian point process results
for coalescing particles require bidirectional
dynamics~\cite{TZ2011,GarrodPTZ2018}.

\subsubsection{Poisson jumps}

With jump probability~$\lambda \epsilon$ and step
size~$\epsilon \to 0$, the forward particles
perform independent Poisson jumps of $+1$ at
rate~$\lambda$ and coalesce upon meeting. The
backward transition weight becomes
\[
w_t(x, y) = e^{-\lambda t}\,
  \frac{(\lambda t)^{x-y}}{(x-y)!}
\]
for $x \geq y$. For backward particles separated
by $\Delta = x_J - x_I$, the Karlin--McGregor
complement gives
\[
A_{IJ} = 1 - e^{-2\lambda t}
  \Bigl[I_0(2\lambda t) +
  2\!\sum_{k=1}^{\Delta-1} I_k(2\lambda t) +
  I_\Delta(2\lambda t)\Bigr],
\]
where $I_k$ is the modified Bessel function of the
first kind. To derive this, observe that the
separation $D(t) = X_J(t) - X_I(t)$ of two
independent backward particles is a symmetric
continuous-time random walk on~$\ZZ$ starting
at~$\Delta$, making $\pm 1$ jumps at rate~$\lambda$
each. Its marginal distribution
is~\cite{Skellam1946}
$\PP(D(t) = k) = e^{-2\lambda t}\,
I_{|k-\Delta|}(2\lambda t)$.
Since the walk is skip-free and symmetric, the
reflection principle gives
\[
A_{IJ} = 2\,\PP(D(t) < 0) + \PP(D(t) = 0).
\]
Indeed, every path from~$\Delta$ to~$k > 0$ that
touches~$0$ bijects with a path from~$-\Delta$
to~$k$, which by symmetry has the same weight as a
path from~$\Delta$ to~$-k$. Summing the tail and
using $\sum_{k=-\infty}^{\infty} \PP(D(t) = k) = 1$
gives the Bessel formula.

\subsection{Bidirectional Poisson coalescence}
\label{sec:bidirectional-poisson}
\label{ex:bidirectional-poisson}
\label{sec:continuous-time}
\label{ex:continuous-time}

In coordinates $t = u + v$,
$x = u - \lfloor t/2 \rfloor$, a boundary particle
jumps $+1$ on even time steps and $-1$ on odd time
steps. With alternating edge weights
$p_{u,v} = \lambda_+ \epsilon$ on even diagonals
and $p_{u,v} = 1 - \lambda_- \epsilon$ on odd
diagonals, the $\epsilon \to 0$ limit gives
coalescing particles with independent Poisson jumps
of $+1$ at rate~$\lambda_+$ and $-1$ at
rate~$\lambda_-$. The backward particle has swapped
rates ($+1$ at rate~$\lambda_-$, $-1$ at
rate~$\lambda_+$), giving the transition weight
\[
w_t(x, y) = e^{-\lambda t}
  \left(\frac{\lambda_-}{\lambda_+}\right)
  ^{\!\!(y-x)/2}
  I_{|y - x|}\left(2\sqrt{\lambda_+ \lambda_-}\ t\right),
\]
where $\lambda = \lambda_+ + \lambda_-$ is the
total jump rate. The crossing-or-meeting probability
depends only on~$\lambda$, not on the individual
rates: the difference of two independent particles
is a symmetric process regardless of the drift.
The Karlin--McGregor complement gives
\[
A_{IJ} = 1 - e^{-2\lambda t}
  \Bigl[I_0(2\lambda t) +
  2\!\sum_{k=1}^{\Delta-1} I_k(2\lambda t) +
  I_\Delta(2\lambda t)\Bigr]
\]
with $\Delta = x_J - x_I$. This coincides with the
totally asymmetric formula
(\Cref{sec:asymmetric-walk}): in both cases the
difference of two independent backward particles
performs the same symmetric random walk.

More generally, the rates $\lambda_+$ and~$\lambda_-$
may vary with position and time,
modeling the continuous-time voter model with
inhomogeneous copying rates. The Pfaffian formula
(\Cref{thm:pfaffian-empty-interval}) still applies;
the transition weight~$w_t$ must then be computed
for the inhomogeneous rates along each backward
path.

\appendix

\section{CLT for sums of indicators via idempotence}
\label{sec:appendix-clt}

This appendix proves a central limit theorem for sums
of dependent $0$-$1$ random variables, using only
the joint cumulants of \emph{distinct} indices.
The key simplification is that indicator variables
are idempotent ($X_i^2 = X_i$), which collapses
repeated-index cumulants to distinct-index ones.
The ingredients are individually well known
(multilinearity of cumulants, the idempotent
multinomial identity, and the method of cumulants).
In the continuous setting, the analogous reduction
is classical: factorial cumulant measures
(see Daley and
Vere-Jones~\cite[Chapter~5]{DaleyVereJones2008})
are defined on tuples of distinct points, and since
the diagonal has measure zero in $\RR^k$, ordinary
and factorial cumulants of a count coincide.
In the discrete setting, however, repeated indices
contribute nontrivially, and the idempotence
reduction below handles them; we include the short
proof for completeness.

\subsection{The CLT criterion}
\label{sec:clt-criterion}

\begin{proposition}
\label{prop:clt-indicators}
For each $n \geq 1$, let $X_1, \ldots, X_n$ be $0$-$1$
random variables (possibly dependent), and let
$S_n = X_1 + \cdots + X_n$.
Write $\kappa_A = \kappa(X_i : i \in A)$ for the joint
cumulant indexed by a set~$A \subseteq [n]$ of distinct
indices. Assume:
\begin{enumerate}[label=(\roman*)]
\item \label{it:summability}
  For each fixed $L \geq 1$,
  \[
  \sup_{i \in [n]}\;
    \sum_{\substack{A \subseteq [n] \\
      i \in A,\; |A| = L}}
    |\kappa_A| = O(1),
  \]
  where the implied constant may depend on~$L$ but not
  on~$n$.
\item \label{it:variance}
  $\Var(S_n) \asymp n$.
\end{enumerate}
Then for each fixed $m \geq 1$,
$\kappa_m(S_n) = O(n)$, and
\[
\frac{S_n - \EE[S_n]}{\sqrt{\Var(S_n)}}
\;\xrightarrow{d}\; \mathcal{N}(0,1)
\qquad \text{as } n \to \infty.
\]
\end{proposition}

\subsection{Proof via idempotence}
\label{sec:proof-idempotence}

By multilinearity, $\kappa_m(S_n)$ is a sum over
tuples $(\alpha_1, \ldots, \alpha_m) \in [n]^m$.
Idempotence ($X_i^2 = X_i$) ensures that every
repeated-index cumulant can be expressed as a
polynomial in the distinct-index cumulants~$\kappa_A$
and the marginal probabilities $p_i = \EE X_i$.
The following example illustrates this reduction
for $m = 2$ and~$3$; although the resulting formulas
are explicit, they grow rapidly in complexity with~$m$,
and we do not attempt to find closed forms.
Instead, \Cref{lem:cumulant-bound} below provides
the bound $\kappa_m(S_n) = O(n)$ that suffices for
the CLT, bypassing the need for exact expressions.

\begin{example}[Small cases]
\label{ex:small-cumulants}
Write $p_i = \EE X_i$.

For $m = 2$, multilinearity gives
$\kappa_2(S_n) = \sum_{i,j} \Cov(X_i, X_j)$.
On the diagonal, idempotence
($\EE X_i^2 = p_i$) yields
$\Cov(X_i, X_i) = p_i(1 - p_i)$, so
\[
\kappa_2(S_n)
= \sum_i p_i(1 - p_i)
  + \sum_{i \neq j} \Cov(X_i, X_j).
\]

For $m = 3$, the Leonov--Shiryaev formula
and idempotence reduce every repeated-index
cumulant to lower-order ones:
\begin{align*}
\kappa(X_i, X_i, X_k)
  &= (1 - 2p_i)\,\Cov(X_i, X_k),
  \\
\kappa_3(X_i)
  &= p_i(1-p_i)(1-2p_i).
\end{align*}
Thus
\begin{align*}
\kappa_3(S_n)
&= \sum_{\substack{i,j,k \\ \text{distinct}}}
     \kappa(X_i, X_j, X_k)
\\
&\quad
  + 3 \sum_{i \neq j}
     (1 - p_i - p_j)\,\Cov(X_i, X_j)
\\
&\quad
  + \sum_i p_i(1-p_i)(1-2p_i).
\end{align*}
Every term is determined by the
distinct-index cumulants $\kappa_A$ and
the marginal probabilities~$p_i$.
\end{example}

\begin{lemma}[Cumulant bound for idempotent variables]
	\label{lem:cumulant-bound}
	Under assumption~\ref{it:summability} of
	\Cref{prop:clt-indicators}, for each
	fixed~$m \geq 1$ define
	\[
	\kappa_m^{*,i}(n) \;:=\;
	\sum_{\alpha_2, \ldots, \alpha_m \in [n]}
	|\kappa(X_i, X_{\alpha_2}, \ldots, X_{\alpha_m})|.
	\]
	Then
	\[
		\sup_{i \in [n]} \kappa_m^{*,i}(n) = O(1),
	\]
	where the implied constant depends on~$m$ but not
	on~$n$.
\end{lemma}

\begin{proof}[Proof of \Cref{lem:cumulant-bound}]
We proceed by induction on~$m$.
The base case $m = 1$ is immediate:
$\kappa_1^{*,i}(n) = |\EE X_i| \leq 1$ for
every~$i$.

For the inductive step, fix~$i \in [n]$ and bound
$\kappa_m^{*,i}(n)$.
Cover the tuples $(\alpha_2, \ldots, \alpha_m)$ by
$\binom{m}{2} + 1$ classes (not necessarily
disjoint). For each pair of positions
$1 \leq p < q \leq m$, let
\[
T_{pq} = \bigl\{(\alpha_2, \ldots, \alpha_m)
  \in [n]^{m-1}
  : \text{positions } p \text{ and } q
  \text{ carry the same value}\bigr\},
\]
where position~$1$ carries the fixed value~$i$.
Every tuple with at least one repeated index
belongs to some~$T_{pq}$, so
\begin{equation}\label{eq:class-bound}
\kappa_m^{*,i}(n)
\;\leq\;
\underbrace{(m{-}1)!
  \sum_{\substack{A \subseteq [n],\;
  i \in A \\ |A|=m}} |\kappa_A|}_{%
  \text{all distinct: } O(1)}
\;+\;
\sum_{1 \leq p < q \leq m}
  T_{pq}^{*,i},
\end{equation}
where $T_{pq}^{*,i}$ is the sum of
$|\kappa(X_i, X_{\alpha_2}, \ldots,
X_{\alpha_m})|$ over $T_{pq}$.
The all-distinct term is~$O(1)$ by
assumption~\ref{it:summability}. It remains to show
$T_{pq}^{*,i} = O(1)$ for each pair~$(p,q)$.

\medskip

\emph{Idempotence reduction.}
Fix a pair~$(p,q)$ and a tuple
$(\alpha_2, \ldots, \alpha_m) \in T_{pq}$.
Write $r = \alpha_p = \alpha_q$ for the repeated
value (with the convention $\alpha_1 = i$), and
write $Z_1, \ldots, Z_{m-2}$ for the variables
$X_{\alpha_s}$, $s \in [m] \setminus \{p,q\}$,
in increasing order of~$s$.
(When $p \geq 2$, the pinned variable~$X_i$
is one of the~$Z_j$.)
The $m$-argument cumulant reads
\[
\kappa(X_{\alpha_1}, \ldots, X_{\alpha_m})
= \kappa(X_r, X_r, Z_1, \ldots, Z_{m-2}),
\]
since $X_r$ occupies both positions~$p$ and~$q$.

The Leonov--Shiryaev product
formula~\cite[Theorem~1]{LeonovShiryaev1959}
(see also Speed~\cite[Proposition~4.3]{Speed1983})
gives
\begin{equation}\label{eq:product-formula}
\kappa(X_r \cdot X_r,\, Z_1, \ldots, Z_{m-2})
= \kappa(X_r, X_r, Z_1, \ldots, Z_{m-2})
+ R,
\end{equation}
where
\begin{equation}\label{eq:R-def}
R = \sum_{S \sqcup T = [m-2]}
  \kappa(X_r,\, Z_s : s \in S)
  \;\cdot\;
  \kappa(X_r,\, Z_t : t \in T).
\end{equation}
The sum runs over all $2^{m-2}$ ordered partitions
of~$[m-2]$ into two blocks (either block may be
empty, giving $\kappa_1 = \EE$).
Every term in~$R$ is a product of two cumulants
whose orders sum to~$m$ and are each
strictly less than~$m$.

By idempotence ($X_r^2 = X_r$), the left-hand
side of~\eqref{eq:product-formula} equals
$\kappa(X_r, Z_1, \allowbreak \ldots, Z_{m-2})$.
Rearranging:
\begin{equation}\label{eq:ls-step}
\kappa(X_r, X_r, Z_1, \ldots, Z_{m-2})
= \kappa(X_r, Z_1, \ldots, Z_{m-2}) - R.
\end{equation}
This is the key step: the $m$-argument cumulant
(left) equals an $(m{-}1)$-argument cumulant minus
the remainder~$R$.
Taking absolute values and summing
over all tuples
$(\alpha_2, \ldots, \alpha_m) \in T_{pq}$:
\[
T_{pq}^{*,i}
\;\leq\;
\sum_{(\alpha_2, \ldots, \alpha_m) \in T_{pq}}
  |\kappa(X_r,\, Z_1, \ldots, Z_{m-2})|
\;+\;
\sum_{(\alpha_2, \ldots, \alpha_m) \in T_{pq}}
  |R|.
\]

\medskip

\emph{Bounding the first sum.}
The $(m{-}1)$-argument cumulant
$\kappa(X_r, Z_1, \ldots, Z_{m-2})$
has $m - 2$ free indices (the constraint
$\alpha_p = \alpha_q$ eliminates one).

If $p = 1$ (so $r = i$), the cumulant is
$\kappa(X_i, Z_1, \ldots, Z_{m-2})$;
summing over the $m - 2$ free indices gives
at most $\kappa_{m-1}^{*,i}(n) = O(1)$
by induction.

If $p \geq 2$, the pinned variable~$X_i$ is one
of the~$Z_j$, so the cumulant still
involves~$X_i$; summing over the $m - 2$ free
indices (including~$r$) again gives
at most $\kappa_{m-1}^{*,i}(n) = O(1)$.

\medskip

\emph{Bounding the remainder sum.}
Each term in~$R$ (see~\eqref{eq:R-def}) is
\[
|\kappa(X_r,\, Z_s : s \in S)|
\;\cdot\;
|\kappa(X_r,\, Z_t : t \in T)|
\]
with $S \sqcup T = [m-2]$.
Both factors contain~$X_r$; the pinned
variable~$X_i$ sits in one factor
(or equals~$X_r$ when~$p = 1$).

\emph{Case $p = 1$} ($r = i$):
both factors have~$X_i$ as their first argument,
and the free indices in~$S$ and~$T$ are disjoint.
The sum factors:
\[
\sum_{(\alpha_2, \ldots, \alpha_m) \in T_{1q}}
  \!\!\!
  |\kappa(X_i, Z_S)| \cdot |\kappa(X_i, Z_T)|
= \kappa_{|S|+1}^{*,i} \cdot \kappa_{|T|+1}^{*,i}
= O(1)
\]
by the inductive hypothesis
(both orders are~$< m$).

\emph{Case $p \geq 2$}:
$X_i$ is one of the~$Z_j$; say
$Z_{j_0} = X_i$ belongs to the~$S$-factor.
Write $S_0 = S \setminus \{j_0\}$, so the
first factor is
$\kappa(X_r, X_i, Z_s : s \in S_0)$
and the second is
$\kappa(X_r, Z_t : t \in T)$.
Sum the second factor over its free indices
(for each fixed~$r$) to get at most
$\sup_{j} \kappa_{|T|+1}^{*,j}(n) = O(1)$;
then sum the first factor over its free indices
(including~$r$) to get at most
$\kappa_{|S_0|+2}^{*,i}(n) = O(1)$.
The product is~$O(1)$.

Since~$R$ has $2^{m-2}$ terms (a constant
depending only on~$m$),
$\sum_{T_{pq}} |R| = O(1)$.

Combining: $T_{pq}^{*,i} = O(1)$ for each
pair~$(p,q)$, uniformly in~$i$,
and~\eqref{eq:class-bound} gives
$\sup_i \kappa_m^{*,i}(n) = O(1)$.
\end{proof}

\begin{proof}[Proof of \Cref{prop:clt-indicators}]
By multilinearity of cumulants,
\[
\kappa_m(S_n)
= \sum_{\alpha_1, \ldots, \alpha_m \in [n]}
  \kappa(X_{\alpha_1}, \ldots, X_{\alpha_m}).
\]
Taking absolute values and grouping by~$\alpha_1$,
\[
|\kappa_m(S_n)|
\;\leq\;
\sum_{i \in [n]} \kappa_m^{*,i}(n)
\;=\; O(n)
\]
by \Cref{lem:cumulant-bound}.
For the standardized variable
$\tilde{S}_n = (S_n - \EE S_n)/\sqrt{\Var S_n}$:
\[
\kappa_m(\tilde{S}_n)
= \frac{\kappa_m(S_n)}{(\Var S_n)^{m/2}}
= O(n^{1-m/2}) \to 0
\qquad \text{for } m \geq 3,
\]
while $\kappa_1(\tilde{S}_n) = 0$ and
$\kappa_2(\tilde{S}_n) = 1$. Since all cumulants of
order~$\geq 3$ vanish, the method of cumulants gives
$\tilde{S}_n \xrightarrow{d} \mathcal{N}(0,1)$.
\end{proof}

\subsection{Remarks}
\label{sec:clt-remarks}

\begin{remark}[Translation-invariant processes]
\label{rem:translation-invariance}
For translation-invariant indicators
(e.g., a stationary point process
on~$\ZZ \cap [1, n]$),
assumption~\ref{it:summability} follows from
the weaker global bound
$\sum_{|A| = L} |\kappa_A| = O(n)$:
by stationarity, each~$i$ contributes equally
to $\sum_i \sum_{A \ni i,\, |A| = L}
|\kappa_A| = L \sum_{|A|=L} |\kappa_A|$,
so the per-site sum is $O(1)$.
\end{remark}

\begin{remark}[Berry--Esseen bound]
\label{rem:berry-esseen}
The bound
$|\kappa_m(S_n)| = O(n)$
from \Cref{lem:cumulant-bound} gives
$|\kappa_m(\tilde{S}_n)| = O(n^{1-m/2})$
for the standardized sum. Since $|X_i| \leq 1$,
this satisfies the
\emph{Statulevi\v{c}ius condition}
$|\kappa_j(\tilde{S}_n)|
\leq j!^{1+\gamma}/\Delta^{j-2}$
with $\Delta \asymp \sqrt{n}$ and suitable~$\gamma$
(the boundedness of the indicators ensures
$\gamma < \infty$).
This yields a Berry--Esseen bound
$\sup_x |F_{\tilde{S}_n}(x) - \Phi(x)|
= O(n^{-1/2})$
via Corollary~2.1 of Saulis and
Statulevi\v{c}ius~\cite{SaulisStatulevicius1991},
and moderate deviation estimates via their
Lemma~2.3 (due to Rudzkis, Saulis, and
Statulevi\v{c}ius~\cite{RudzkisSaulisStatulevicius1978});
see also D\"oring, Jansen, and
Schubert~\cite{DoringJansenSchubert2022},
Sections~2.1--2.2.
\end{remark}

\begin{remark}[Role of idempotence]
\label{rem:idempotence}
The inductive step~\eqref{eq:ls-step} uses
idempotence ($X_r^2 = X_r$) to replace
$X_r \cdot X_r$ on the left-hand side
of~\eqref{eq:product-formula} by a single~$X_r$,
reducing the $m$-argument cumulant to an
$(m{-}1)$-argument one. For general bounded random
variables, $\kappa(X_r \cdot X_r, Z_1, \ldots)$
would be a genuinely different object from any
ordinary cumulant, and the induction would not
close without stronger summability assumptions.
\end{remark}

\section{Nonzero colorings for $k = 3$}
\label{sec:appendix-colorings}

For $k = 3$, there are $90$ colorings (sequences of
length~$6$ over $\{1, 2, 3\}$ with each color used
exactly twice), of which $14$ have
$c(I) \neq 0$ (\Cref{thm:coloring-formula}).
Each diagram below shows the six simplex positions
$y_1 < \cdots < y_6$ (top) connected to the three
walls $1, 2, 3$ (bottom) according to the
coloring~$I$.
Wall~$1$: solid blue; wall~$2$: dashed red;
wall~$3$: dotted green.

The \emph{mirror symmetry} reverses the coloring
sequence and relabels $1 \leftrightarrow 3$,
$2 \leftrightarrow 2$.
The coefficient~$c(I)$ is invariant under this
symmetry.  The $14$ nonzero colorings split into
$4$~self-mirror colorings and $5$~mirror pairs.

\subsection{The unique negative coloring ($c = -8$)}

The only coloring with $c(I) < 0$ is self-mirror:

\bigskip

\begin{center}
\coloringdiagram{2}{1}{3}{1}{3}{2}{-8}
\end{center}

\subsection{Positive colorings ($c = +8$)}

The remaining $13$ nonzero colorings all have
$c(I) = +8$.

\subsubsection{Self-mirror}

\bigskip

\begin{center}
\coloringdiagram{2}{3}{1}{3}{1}{2}{+8}%
\qquad\qquad
\coloringdiagram{3}{2}{1}{3}{2}{1}{+8}%
\end{center}

\bigskip

\begin{center}
\coloringdiagram{3}{2}{3}{1}{2}{1}{+8}%
\end{center}

\subsubsection{Mirror pairs}

\bigskip

\begin{center}
\coloringdiagram{2}{1}{3}{2}{3}{1}{+8}%
\quad$\longleftrightarrow$\quad
\coloringdiagram{3}{1}{2}{1}{3}{2}{+8}%
\end{center}

\bigskip

\begin{center}
\coloringdiagram{2}{3}{1}{2}{3}{1}{+8}%
\quad$\longleftrightarrow$\quad
\coloringdiagram{3}{1}{2}{3}{1}{2}{+8}%
\end{center}

\bigskip

\begin{center}
\coloringdiagram{2}{3}{1}{3}{2}{1}{+8}%
\quad$\longleftrightarrow$\quad
\coloringdiagram{3}{2}{1}{3}{1}{2}{+8}%
\end{center}

\bigskip

\begin{center}
\coloringdiagram{2}{3}{2}{1}{3}{1}{+8}%
\quad$\longleftrightarrow$\quad
\coloringdiagram{3}{1}{3}{2}{1}{2}{+8}%
\end{center}

\bigskip

\begin{center}
\coloringdiagram{3}{1}{2}{3}{2}{1}{+8}%
\quad$\longleftrightarrow$\quad
\coloringdiagram{3}{2}{1}{2}{3}{1}{+8}%
\end{center}

\section*{Acknowledgments}

We thank Theodoros Assiotis, Bal\'azs B\'ar\'any,
Maciej Do\l{}ęga, Sho Matsumoto, B\'alint T\'oth,
\'Akos Urb\'an, Oleg Zaboronski, and Karol \.Zyczkowski
for stimulating discussions and helpful literature
suggestions.

We thank Richard Arratia for generously providing
access to his PhD thesis~\cite{Arratia1979}.

Claude Code (Anthropic) was used as an assistant during
manuscript preparation.

\printbibliography

\end{document}